\documentclass[smallextended]{svjour3}

\usepackage{amsmath,amssymb,bbm,color,psfrag,subfigure,amsfonts}
\smartqed
\usepackage{graphicx}                       
\usepackage{enumerate}
\usepackage[normalem]{ulem}

\newcommand{\map}[3]{#1: #2 \rightarrow #3}
\newcommand{\setdef}[2]{\left\{#1 \; | \; #2\right\}}

\newcommand{\mc}{\mathcal}
\newcommand{\onebf}{\mathbf{1}}

\newcommand{\de}{\mathrm{d}}

\newcommand{\RR}{\mathbb{R}}

\newcommand{\NN}{\mathbb{N}}

\newcommand{\ba}{\begin{array}}
\newcommand{\ea}{\end{array}}
\newcommand{\until}[1]{\{1,\dots,#1\}}

\newcommand{\ds}{\displaystyle}
\newcommand{\be}{\begin{equation}}
\newcommand{\ee}{\end{equation}}

    \newcommand{\zerobf}{\mathbf{0}}
    \newcommand{\sgn}[1]{\mathrm{sgn}\left(#1\right)}
 \newcommand{\ksmargin}[1]{\marginpar{\color{red}\footnotesize [KS]:
#1}}

\newcommand{\ymax}{y_{\mathrm{max}}}

\newcommand{\ymin}{y_{\mathrm{min}}}

  \newcommand{\supp}{\text{supp}}
\newcommand{\lambdamax}{\lambda_{\text{max}}}
\graphicspath{{../../../../../fig/}}

\title{
Service Rate, Busy Period \& Throughput Analysis of a Horizontal Traffic Queue
}

\author{Mohammad Motie \and Ketan Savla
}

\institute{M. Motie \at
              Sonny Astani Department of Civil and Environmental Engineering\\
              University of Southern California \\
              \email{motiesha@usc.edu}           
           \and
           K. Savla \at
           Sonny Astani Department of Civil and Environmental Engineering\\
              University of Southern California \\
              \email{ksavla@usc.edu}  
}

\date{\today}

\begin{document}

\date{}

\maketitle
\thispagestyle{empty}

\begin{abstract}

We consider a horizontal traffic queue (HTQ) on a periodic road segment, where vehicles arrive according to a spatio-temporal Poisson process, and depart after traveling a distance that is sampled independently and identically from a spatial distribution. When inside the queue, the speed of a vehicle is proportional to a power $m >0$ of the distance to the vehicle in front. The service rate of HTQ is equal to the sum of the speeds of the vehicles, and has a complex dependency on the state (vehicle locations) of the system. We show that the service-rate increases (resp., decreases) in between arrivals and departures for $m<1$ (resp., $m>1$) case. 
For a given initial condition, we define the throughput of such a queue as the largest arrival rate under which the queue length remains bounded. We extend the busy period calculations for M/G/1 queue to our setting, including for non-empty initial condition. These calculations are used to prove that the throughput for $m=1$ case is equal to the inverse of the 
time required to travel average total distance by a solitary vehicle in the system, and also to derive a probabilistic upper bound on the queue length over a finite time horizon for the $m>1$ case. Finally, we study throughput under a release control policy, where the additional expected waiting time caused by the control policy is interpreted as the magnitude of the perturbation to the arrival process. We derive a lower bound on throughput for a given  
combination of maximum allowable perturbation, for $m<1$ and $m>1$ cases. In particular, if the allowable perturbation is sufficiently large, then this lower bound grows unbounded as $m \to 0^+$. Illustrative simulation results are also presented.

\end{abstract}
\section{Introduction}
\label{sec:introduction}
We consider a horizontal traffic queue (HTQ) on a periodic road segment, where vehicles arrive according to a spatio-temporal Poisson process, and depart the queue after traveling a distance that is sampled independently and identically from a spatial distribution. When inside the queue, the speed of a vehicle is proportional to a power $m>0$ of the distance to the vehicle in front. For a given initial condition, we define the throughput of such a queue as the largest arrival rate under which the queue length remains bounded. We provide rigorous analysis for the service rate, busy period distribution, and throughput of the proposed HTQ. 

Our motivation for studying HTQ comes from advancements in connected and autonomous vehicle technologies that allow to program individual vehicles with rules that can optimize system level performance. Within this application context, one can interpret the results of this paper as rigorously characterizing the impact of a parametric class of car-following behavior on system throughput. 


In the linear case ($m=1$), i.e., when the speed of every vehicle is proportional to the distance to the vehicle directly in front, the periodicity of the road segment implies that the sum of the speeds of the vehicles is proportional to the total length of the road segment, i.e., it is constant. This feature allows us to exploit the equivalence between workload and queue length to show that, independent of the initial condition and almost surely, the throughput is the inverse of the time required by a solitary vehicle to travel average distance. 

In the non-linear case ($m \neq 1$), the cumulative service rate of HTQ queue is constant if and only if all the inter-vehicle distances are equal. For all other inter-vehicle configurations, we show that the service rate is strictly decreasing (resp., strictly increasing) in the super-linear, i.e., $m>1$ (resp., sub-linear, i.e., $m<1$) case. The service rate exhibits an another contrasting behavior in the sub- and super-linear regimes. In the super-linear case, the service rate is maximum (resp., minimum) when all the vehicles are co-located (resp., when the inter-vehicle distances are equal), and vice-versa for the sub-linear case. Using a combination of these properties, we prove that, when the length of the road segment is at most one, the throughput in the super-linear (resp., sub-linear) case is upper (resp., lower) bounded by the throughput for the linear case.

We prove the remaining bounds on the throughput for the non-linear case as follows. The standard calculations for joint distributions of duration and number of arrivals during a busy period for M/G/1 queue are extended to the HTQ setting, including for non-empty initial conditions. 
These joint distributions are used to derive probabilistic upper bounds on queue length over finite time horizons for HTQ for the $m>1$ case. Such bounds are optimized to get lower bounds on throughput defined over finite time horizons. Simulation results show good comparison between such lower bounds and numerical estimates.
 
We also analyze throughput in the sub-linear and super-linear cases under perturbation to the arrival process, which is attributed to the additional expected waiting time induced by a release control policy that adds appropriate delay to the arrival times to ensure a desired minimum inter-vehicle distance $\triangle>0$ at the time of a vehicle joining the HTQ. Since the minimum inter-vehicle distance is non-decreasing in between arrivals and jumps, this implies an upper bound on the queue length which is inversely proportional to $\triangle$. We derive a lower bound on throughput for a given  
combination of maximum allowable perturbation. In particular, if the allowable perturbation is sufficiently large, then this lower bound grows unbounded, as $m \to 0^+$.  


Queueing models have been used to model and analyze traffic systems. The focus here has been primarily on vertical queues, under which vehicles travel at maximum speed until they hit a congestion spot where all vehicles queue on top of each other.
%
The queue length and waiting time of a minor traffic stream at an unsignalized intersection where major traffic stream has high priority is studied in \cite{Tanner:62} and \cite{Heidemann:91}. 
In \cite{Heidemann:94}, a vertical single server queue is utilized to model the queue length distribution at signalized intersections.
In \cite{Jain.Smith:97}, a state-dependent queuing system is used to model vehicular traffic flow where the service rate depends on the number of vehicles on each road link. 

 On the other hand, the \emph{horizontal traffic queue} terminology has been primarily used to study macroscopic traffic flow, e.g., see \cite{Helbing:03}. While such models capture the macroscopic relationship between traffic flow and density, a rigorous description and analysis of an underlying queue model is lacking. Indeed, to the best of our knowledge, there is no prior work on the analysis of a traffic queue model that explicitly incorporates car-following behavior. 
%
%

The proposed HTQ has an interesting connection with processor sharing (PS) queues, and this connection does not seem to have been documented before.
A characteristic feature of PS queues is that all the outstanding jobs receive service simultaneously, while keeping the total service rate of the server constant. The simplest model is where the service rate for an individual job is equal to $1/N$, where $N$ is the number of outstanding jobs. In our proposed system, one can interpret the road segment as a server simultaneously providing service to all the vehicles, with the service rate of an individual vehicle equal to its speed. This natural analogy between HTQ and PS queues, to the best of our knowledge, was reported for the first time in our recent work~\cite{Motie.Savla.CDC15}.
The $1/N$ rule applied to our setting implies that all the vehicles travel with the same speed. Clearly, such a rule, or even the general discriminatory PS disciplines, e.g., see \cite{Kleinrock:67}, are not applicable to the car following models considered in this paper. Indeed, the proposed HTQ is best described as a state-dependent PS queue. 


In the PS queue literature, the focus has been on the sojourn time and queue length distribution. For example, see \cite{Ott:84} and \cite{Yashkov:83} for M/G/1-PS queue and 
\cite{Grishechkin:94} for G/G/1-PS queue. Fluid limit analysis for PS queue is provided in \cite{Chen.ea:97} and \cite{Gromoll.Puha.ea:02}. 
However, relatively less attention has been paid to the throughput analysis of state-dependent PS queues. 
In \cite{Moyal:08,Kherani.Kumar:02,Chen.Jordan:07}, throughput analysis for state-dependent PS queues is provided, where throughput is defined as the quantity of work achieved by the server per unit of time.
%
Stability analysis for a single server queue with workload-dependent service and arrival rate is provided in \cite{Bambos.Walrand:89} and \cite{Bekker:05}. However, the dependence of service rate on the system state in the HTQ proposed in the current paper is complex, and hence none of these results are readily applicable. 

In summary, there are several novel contributions of the paper. First, we propose a novel horizontal traffic queue and place it in the context of processor-sharing queues and state-dependent queues. We establish monotonicity properties of service rates in between jumps (i.e., arrivals and departures), and derive bounds on change in service rates at jumps. 
Second, we adapt busy period calculations for M/G/1 queue to our current setup, including for non-empty initial conditions. These results allow us to provide tight results for throughput in the linear case, and probabilistic bounds on queue length over finite time horizon in the super-linear case. We also study throughput under a batch release control policy, whose effect is interpreted as a perturbation to the arrival process. We provide lower bound on the throughput for a maximum permissible perturbation for sub- and super-linear cases. In particular, we show that, for sufficiently large perturbation, this lower bound grows unbounded as $m \to 0^+$. It is interesting to compare our analytical results with simulation results, which suggest a sharp transition in the throughput from being unbounded in the sub-linear regime to being bounded in the super-linear regime. While our analytical results do not exhibit such a phase transition yet, their novelty is in providing rigorous estimates of any kind on the throughput of horizontal traffic queues under nonlinear car following models.
%

The rest of the paper is organized as follows. We conclude this section with key notations to be used throughout the paper. The setting for the proposed horizontal traffic queue and formal definition of throughput are provided in Section~\ref{sec:problem-formulation}. Section~\ref{sec:service-rate-monotonicity} contains useful properties on the dynamics in service rate in between and during jumps. Key busy period properties for the M/G/1 queue are extended to the HTQ case in Section~\ref{sec:busy-period}. Throughput analysis is reported in Section~\ref{sec:throughput-analysis}. Simulations are presented in \ref{sec:simulations}. Concluding remarks and directions for future work are presented in Section~\ref{sec:conclusions}. A few technical intermediate results are collected in the appendix. 

\subsection*{Notations}
Let $\RR$, $\RR_+$, and $\RR_{++}$ denote the set of real, non-negative real, and positive real numbers, respectively. 
Let $\NN$ be the set of natural numbers. 
If $x_1$ and $x_2$ are of the same size, then $x_1 \geq x_2$ implies element-wise inequality between $x_1$ and $x_2$. If $x_1$ and $x_2$ are of different sizes, then $x_1 \geq x_2$ implies inequality only between elements which are common to $x_1$ and $x_2$ -- such a common set of elements will be specified explicitly.  
For a set $\mathcal{J}$, let $\text{int}(\mathcal{J})$ and  $|\mathcal{J}|$ denote the interior and cardinality of $\mathcal{J}$, respectively.
Given $a \in\RR$, and $b > 0$, we let $\mod(a,b):=a-\lfloor \frac{a}{b}\rfloor b$. Let $\mc S_N^L$ be the $N-1$-simplex over $L$, i.e., $\mc S_N^L=\setdef{x \in \RR_+^N}{\sum_{i=1}^N x_i = L}$. When $L=1$, we shall use the shorthand notation $\mc S_N$. 
When referring to the set $\until{N}$, for brevity, we let the indices $i=-1$ and $i=N+1$ correspond to $i=N$ and $i=1$ respectively.
 Also, for $p, q \in \mc S_N$, we let
$D(p || q)$ denote the K-L divergence of $q$ from $p$, i.e., $D(p || q):= \sum_{i=1}^N p_i \log\left(p_i/q_i\right)$.
We also define a permutation matrix,  $P^- \in \{0,1\}^{N \times N}$,  as follows:
\begin{align*}
P^-:=\begin{bmatrix}
\zerobf_{N-1}^T & 1 \\
I_{N-1} & \zerobf_{N-1}
\end{bmatrix}
\end{align*}
where  $\zerobf_N$ and $\onebf_N$ stand for vectors of size $N$, all of whose entries are zero and one, respectively. We shall drop $N$ from $\mathbf{0}_N$ and $\mathbf{1}_N$ whenever it is clear from the context.

\section{The Horizontal Traffic Queue (HTQ) Setup}
\label{sec:problem-formulation}
Consider a periodic road segment of length $L$; without loss of generality, we assume it be a circle. Starting from an arbitrary point on the circle, we assign coordinates in $[0,L]$ to the circle in the clock-wise direction (See Figure \ref{fig:htq}). Vehicles arrive on the circle according to a spatio-temporal process: the arrival process $\{A(t),t\geq0\}$, is assumed to be a Poisson process with rate $\lambda>0$, and the arrival locations are sampled independently and identically from a spatial distribution $\varphi$ and mean value $\bar \varphi$. Without loss of generality, let the support of $\varphi$ be $\text{supp}(\varphi)=[0,\ell]$ for some $\ell\in[0,L]$. 
Upon arriving, vehicle $i$ travels distance $d_i$ in a counter-clockwise direction, after which it departs the system. 
The travel distances $\{d_i\}_{i=1}^\infty$ are sampled independently and identically from a spatial distribution $\psi$ with support $[0,R]$ and mean value $\bar \psi$. 
Let the set of $\varphi$ and $\psi$ satisfying the above conditions be denoted by $\Phi$ and $\Psi$ respectively. 
The stochastic processes for arrival times, arrival locations, and travel distances are all assumed to be independent of each other.

\begin{figure}[htb!]
 \centering
 \includegraphics[width=5cm]{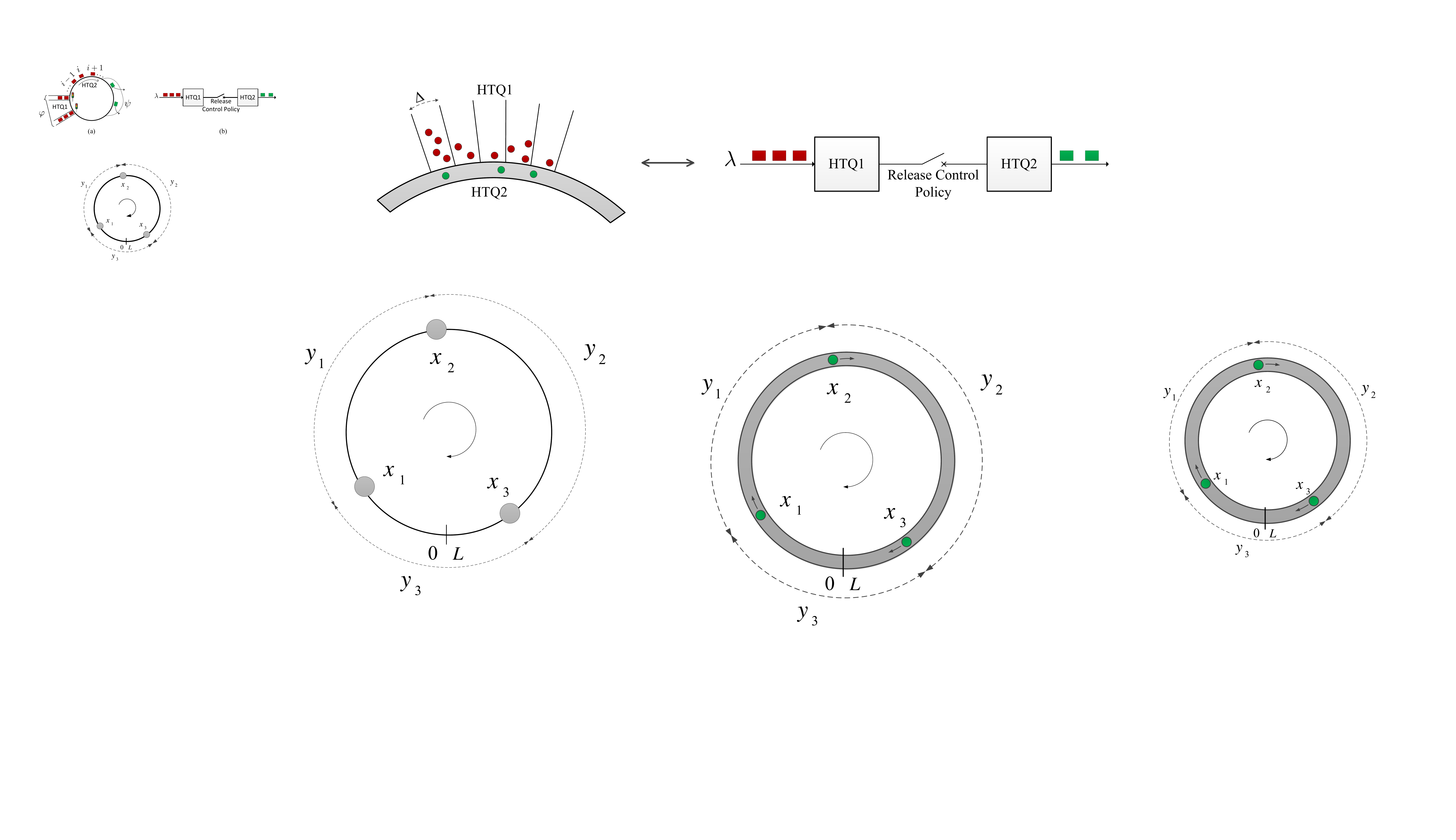} 
  \caption{Illustration of the proposed HTQ with three vehicles.}
    \label{fig:htq}
\end{figure}





\subsection{Dynamics of vehicle coordinates between jumps}
Let the time epochs corresponding to arrival and departure of vehicles be denoted as $\{\tau_1, \tau_2, \ldots\}$. We shall refer to these events succinctly as \emph{jumps}. We now formally state the dynamics under this car-following model. 
We describe the dynamics over an arbitrary time interval of the kind $[\tau_j,\tau_{j+1})$. Let $N \in \NN$ be the fixed number of vehicles in the system during this time interval. 
Define the inter-vehicle distances associated with vehicle coordinates $x \in [0,L]^N$ as follows:
\begin{equation}
\label{eq:inter-vehicle-distance-Rn}
y_i(x)  = \mod \left(x_{i+1} - x_i, L\right), \quad i \in \until{N} 
\end{equation}
where we implicitly let $x_{N+1} \equiv x_1$ (See Figure \ref{fig:htq} for an illustration). Note that the normalized inter-vehicle distances $y/L$ are probability vectors. 
When inside the queue, the speed of every vehicle is proportional to a power $m >0$ of the distance to the vehicle directly in front of it. 
We assume that this power $m > 0$ is the same for every vehicle at all times. Then, starting with $x(\tau_{j}) \in [0,L]^N$, the vehicle coordinates over $[\tau_{j},\tau_{j+1})$ are given by:
\begin{equation}
\label{eq:dynamics-moving-coordinates}
x_i(t) = \mod \left( x_i(\tau_{j}) + \int_{\tau_{j}}^t y_i^m(x(z)) \, dz, L \right), \qquad \forall \, i \in \until{N}, \quad \forall \, t \in [\tau_{j},\tau_{j+1}) \, ,
\end{equation}

\begin{remark}
It is easy to see that the clock-wise ordering of the vehicles is invariant under \eqref{eq:inter-vehicle-distance-Rn}-\eqref{eq:dynamics-moving-coordinates}. 
\end{remark}

The dynamics in inter-vehicle distances is given by:
\begin{equation}
\label{eq:inter-vehicle-distance-dynamics}
\dot{y}_i = y^m_{i+1} - y_i^m, \qquad i \in \until{N} 
\end{equation}
where we implicitly let $y_{N+1} \equiv y_1$.

\subsection{Change in vehicle coordinates during jumps}
Let $x(\tau^-_{j}) =\left(x_1(\tau^-_{j}), \ldots, x_N(\tau^-_{j})\right) \in [0,L]^N$ be the vehicle coordinates just before the jump at $\tau_{j}$. If the jump corresponds to the departure of vehicle $k \in \until{N}$, then the coordinates of the vehicles $x(\tau_{j}) =\left(x_1(\tau_{j}), \ldots, x_{N-1}(\tau_{j})\right) \in [0,L]^{N-1}$ after re-ordering due to the jump, for $i \in \until{N-1}$, are given by:
 \begin{equation*}
		x_i(\tau_{j})=\left\{\begin{array}{ll}
		\ds x_i(\tau^-_{j}) 	& i\in \until{k-1} \\[15pt]
		\ds x_{i+1}(\tau^-_{j})	& i\in \{k+1, \ldots, N-1\}\,.\end{array}\right.
	\end{equation*}

 Analogously, if the jump corresponds to arrival of a vehicle at location $z \in [0,\ell]$ in between the locations of the $k$-th and $k+1$-th vehicles at time $\tau_{j}^-$, then the coordinates of the vehicles $x(\tau_{j}) =\left(x_1(\tau_{j}), \ldots, x_{N+1}(\tau_{j})\right) \in [0,L]^{N+1}$ after re-ordering due to the jump, for $i \in \until{N+1}$, are given by:
 \begin{equation*}
 \begin{split}
 x_{k+1}(\tau_{j}) & = z \\
		x_i(\tau_{j}) & =\left\{\begin{array}{ll}
		\ds x_i(\tau^-_{j}) 	& i\in \until{k} \\[15pt]
		\ds x_{i-1}(\tau^-_{j})	& i\in \{k+2, \ldots, N+1\}\,.\end{array}\right.
		\end{split}
	\end{equation*}

\subsection{Problem statement}
Let $x_0 \in [0,L]^{n_0}$ be the initial coordinates of $n_0$ vehicles present at $t=0$. An HTQ is described by the tuple $\left(L, m, \lambda, \varphi, \psi, x_0\right)$. 
%
%
Let 
$N(t;L, m, \lambda, \varphi, \psi, x_0)$ be the corresponding queue length, i.e., the number of vehicles at time $t$ for an HTQ $\left(L, m, \lambda, \varphi, \psi, x_0\right)$. For brevity in notation, at times, we shall not show the dependence of $N$ on parameters which are clear from the context.


In this paper, our objective is to provide rigorous characterizations of the dynamics of the proposed HTQ. A key quantity that we study is throughput, defined below.  
 
\begin{definition}[Throughput of HTQ]\label{def:throughput}
Given $L > 0$, $m > 0$, $\varphi \in \Phi , \psi \in \Psi$, $x_0 \in [0,L]^{n_0}$, $n_0 \in \NN$ and $\delta \in [0,1)$, 
the throughput of HTQ is defined as:
\begin{equation}
\label{eq:throughput-def}
\lambdamax(L,m,\varphi,\psi,x_0,\delta):= \sup \left\{\lambda \geq 0: \, \Pr \left( N(t;L, m, \lambda, \varphi, \psi, x_0) < + \infty, \quad \forall t \geq 0  \right) \geq 1 - \delta  \right\}.
\end{equation}
\end{definition}


Figure~\ref{fig:throughput-simulations} shows the complex dependency of throughput on key queue parameters such as $m$ and $L$. In particular, it shows that for every $L$, $\varphi$, $\psi$, $x_0$ and $\varphi$, the throughput exhibits a phase transition from being unbounded for $m \in (0,1)$ to being bounded for $m>1$. Moreover, Figure~\ref{fig:throughput-simulations} also suggests that, for sufficiently small $L$, throughput is monotonically non-increasing in $m$, and that it is monotonically non-decreasing in $m>1$, for sufficiently large $L$. Also, it can be observed that initial condition can also affect the throughput. We now develop analytical results that match the throughput profile in Figure~\ref{fig:throughput-simulations} as closely as possible. To that purpose, we will make extensive use of novel properties of \emph{service rate} and \emph{busy period} of the proposed HTQ, which could be of independent interest.

\begin{figure}
\centering
\subfigure[]{\includegraphics[width=6cm]{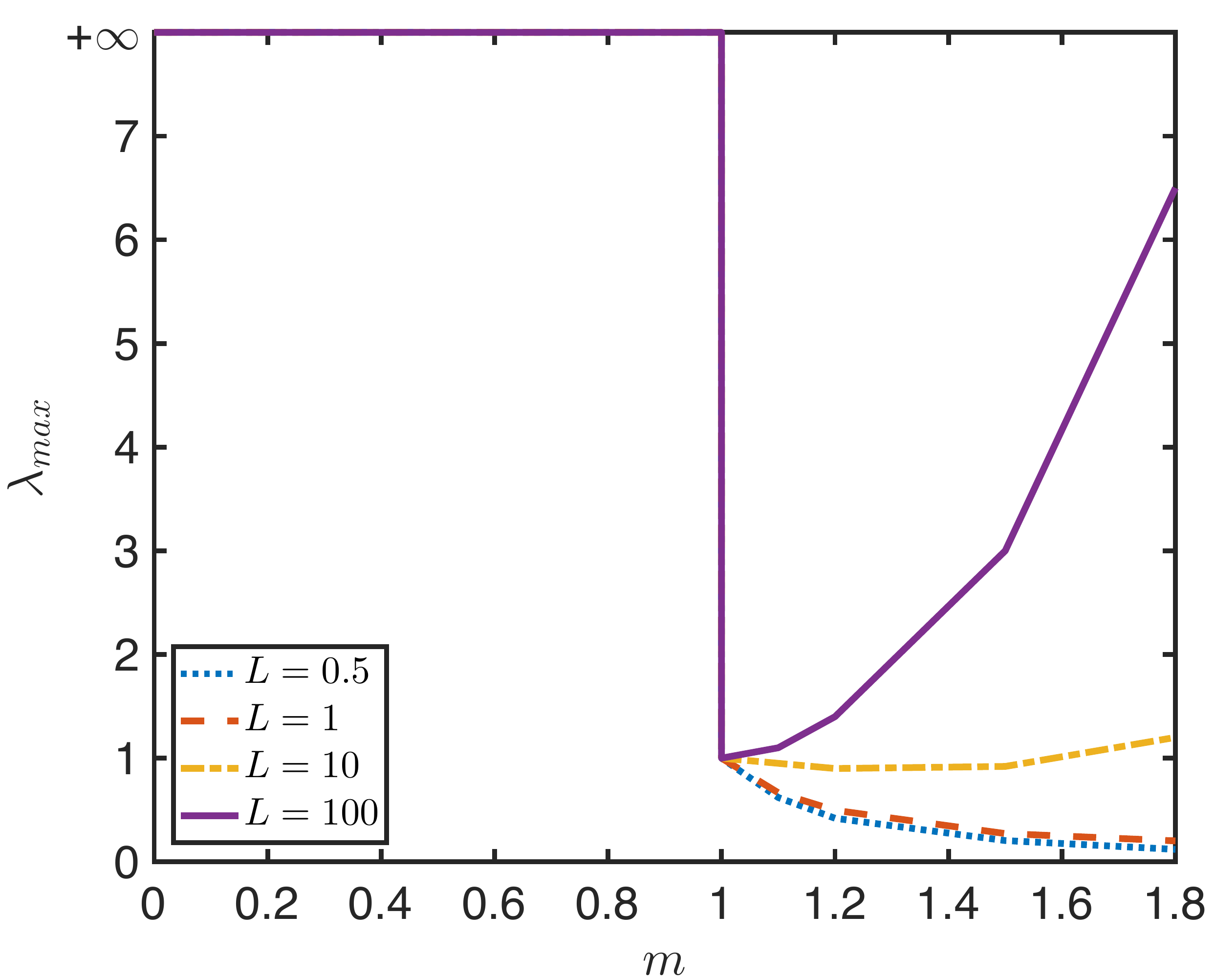}}
\centering
\hspace{0.3in}
\subfigure[]{\includegraphics[width=6cm]{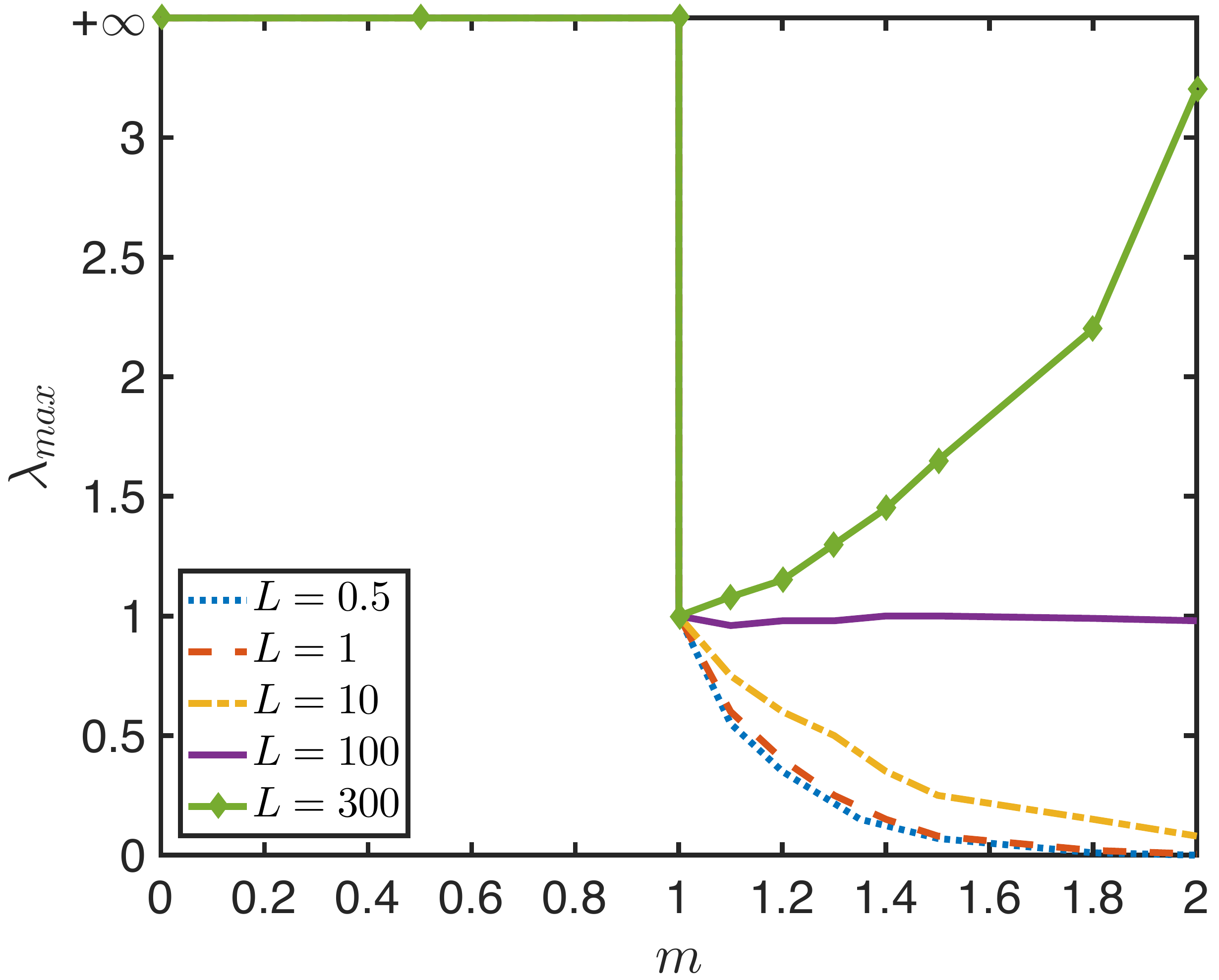}}\\
\subfigure[]{\includegraphics[width=6cm]{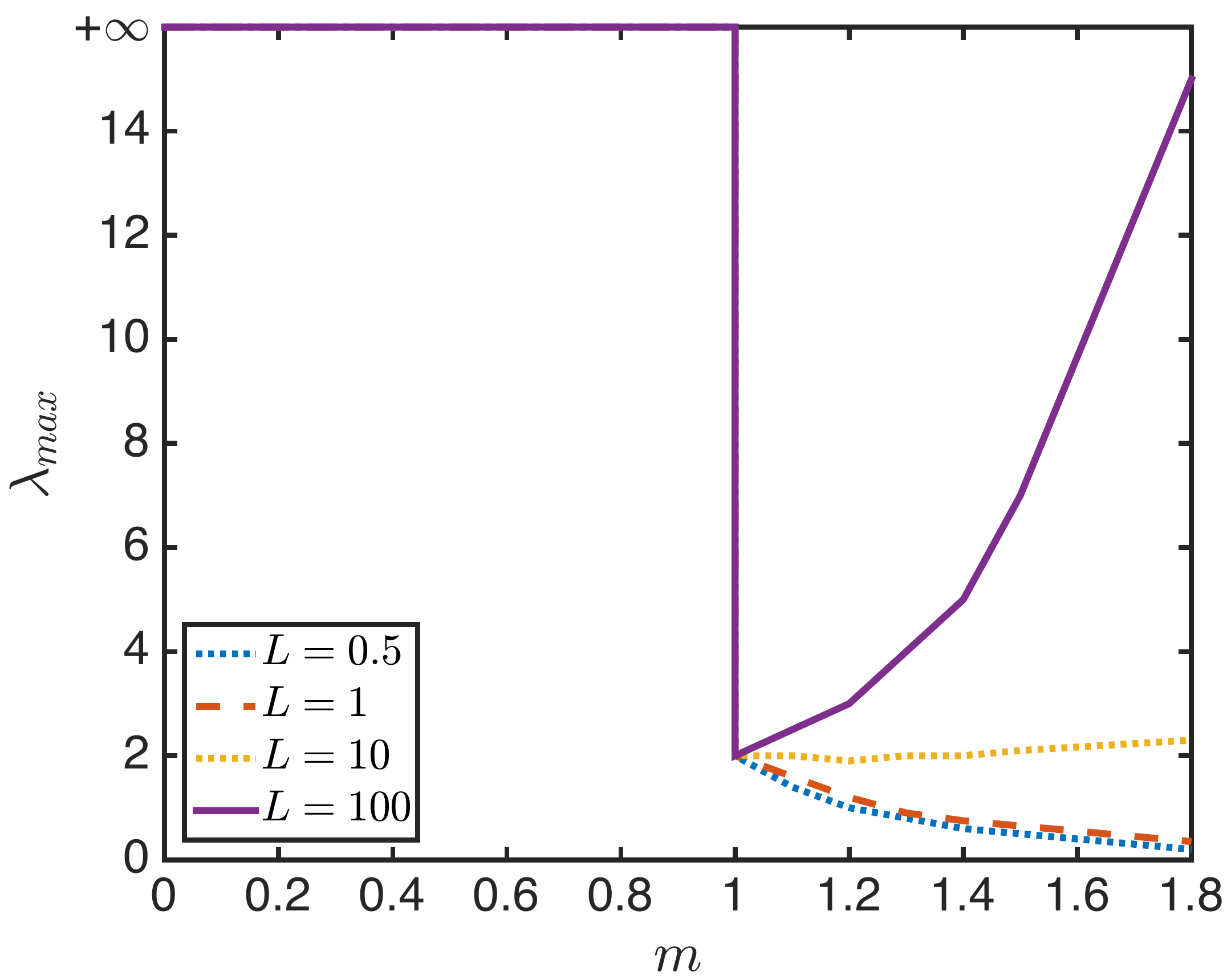}}
\centering
\hspace{0.3in}
\subfigure[]{\includegraphics[width=6cm]{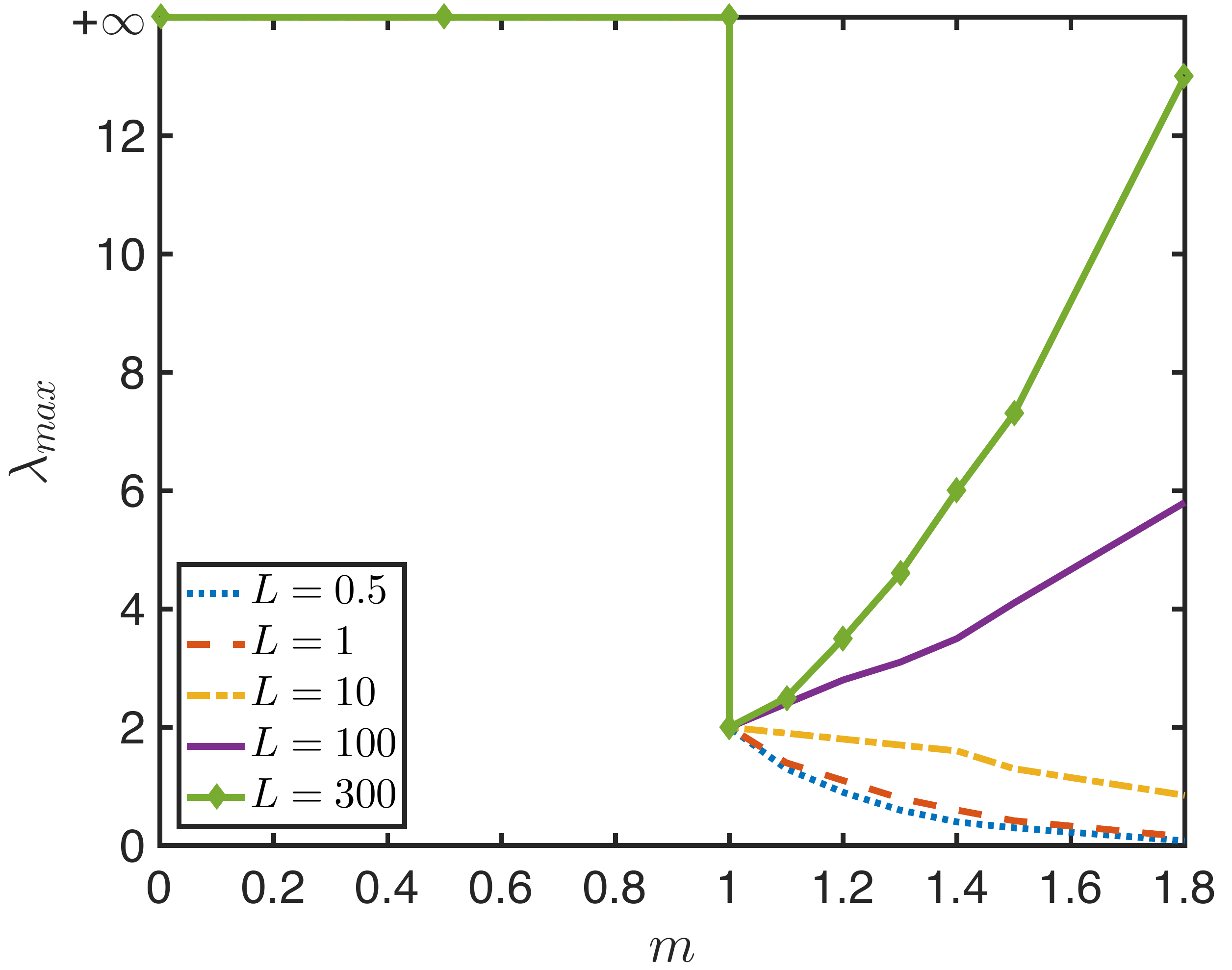}}
\caption{Throughput for various combinations of $m$, $L$, and $n_0$.  The parameters used in individual cases are:  (a) $\varphi=\delta_{0}$, $\psi = \delta_L$, and $n_0=0$ (b) $\varphi=\delta_{0}$, $\psi = \delta_L$, and $n_0=100$ (c) $\varphi=U_{[0,L]}$, $\psi = U_{[0,L]}$, and $n_0=0$ (d) $\varphi=U_{[0,L]}$, $\psi = U_{[0,L]}$, and $n_0=100$. In all the cases, the locations of initial $n_0$ vehicles were chosen at equal spacing in $[0,L]$.
 \label{fig:throughput-simulations}}
\end{figure}




\section{Service Rate Properties of the Horizontal Traffic Queue}
\label{sec:service-rate-monotonicity}
For every $y \in \mc S_N^L$, $N \in \NN$, $L>0$, we let 
$\ymin:=\min_{i \in \until{N}} y_i$, and $\ymax:=\max_{i \in \until{N}} y_i$ denote the minimum and maximum inter-vehicle distances respectively. It is easy to establish the following monotonicity properties of $\ymin$ and $\ymax$.

\begin{lemma}[Inter-vehicle Distance Monotonicity Between Jumps]
\label{lem:vehicle-distance-monotonicity}
For any  $y \in \mc S^L_N$, $N \in \NN$, $L>0$, under the dynamics in \eqref{eq:inter-vehicle-distance-dynamics}, for all $m>0$
$$
\frac{d}{dt} \ymin \geq 0 \qquad \& \qquad \qquad \frac{d}{dt} \ymax \leq 0.
$$
\end{lemma}
\begin{proof}
Let $\ymin(t) = y_j(t)$, i.e., the $j$-th vehicle has the minimum inter-vehicle distance at time $t\geq 0$.
Therefore, \eqref{eq:inter-vehicle-distance-dynamics} implies that $\dot{y}_{\min}(t) = \dot{y}_j(t)=y_{j+1}^m(t) - y_j^m(t)\geq 0$. One can similarly show that $\ymax$ is non-increasing.
\qed
\end{proof}

Due to the complex state-dependence of the departure process, the queue length process is difficult to analyze. We propose to study a related scalar quantity, called \emph{workload} formally defined as follows, where we recall the notations introduced in 
Section~\ref{sec:problem-formulation}.

\begin{definition}[Workload]
The workload associated with the HTQ at any instant is the sum of the distances remaining to be travelled by all the vehicles present at that instant. That is, if the current coordinates and departure coordinates of all vehicles are $x \in [0,L]^N$ and $q \in \RR_+^N$ respectively,  with $q \geq x$, then the workload is given by:
$$
w(x,q):= \sum_{i=1}^N (q_i-x_i).
$$   
\end{definition}

Since the maximum distance to be travelled by any vehicle from the time of arrival to the time of departure is upper bounded by $R$, we have the following simple relationship between workload and queue length at any time instant:
\be\label{eq:workload-upperbound}
w(t) \leq N(t)\, R \, , \qquad \forall \, t \geq 0 \, .
\ee
An implication of \eqref{eq:workload-upperbound} is that unbounded workload implies unbounded queue length in our setting. We shall use this relationship to establish an upper bound on the throughput. 
However, a finite workload does not necessarily imply finite queue length. In order to see this, consider the state of the queue with $N$ vehicles, all of whom have distance $1/N$ remaining to be travelled. Therefore, the workload at this instant is $1/N \times N= 1$, which is independent of $N$. 

When the workload is positive, its rate of decrease is equal to \emph{service rate} in between jumps, defined next.

\begin{definition}[Service Rate]
When the HTQ is not idle, its instantaneous service rate is equal to the sum of the speeds of the vehicles present in the system at that time instant, i.e., $s(x)=\sum_{i=1}^N y_i^m(x)$.
\end{definition}

Since the service rate depends only on the inter-vehicle distances, we shall alternately denote it as $s(y)$. For $m=1$, $s(y)=\sum_{i=1}^N y_i \equiv L$, i.e., the service rate is independent of the state of the system, and is constant in between and during jumps. This property does not hold true in the nonlinear ($m \neq 1$) case. Nevertheless, one can prove interesting properties for the service rate dynamics. We start by deriving bounds on service rate in between jumps. 

\begin{lemma}[Bounds on Service Rates]
\label{lem:service-rate-bounds}
For any $y \in \mc S_N^L$, $N \in \NN$, $L>0$, under the dynamics in \eqref{eq:inter-vehicle-distance-dynamics},
\begin{enumerate}
\item $L^m N^{1-m} \leq s(y) \leq L^m$ if $m > 1$;
\item $L^m \leq s(y) \leq L^m N^{1-m}$ if $m \in (0,1)$.
\end{enumerate}
\end{lemma}
\begin{proof}
Normalizing the inter-vehicular distances by $L$, the service rate can be rewritten as 
\begin{equation}
\label{eq:service-rate-normalized}
s(y)=L^m \sum_{i=1}^N \left(\frac{y_i}{L}\right)^m. 
\end{equation}
Therefore, for $m > 1$, $s(y) \leq L^m \sum_{i=1}^N \frac{y_i}{L}=L^m$. One can similarly show that, for $m \in (0,1)$, $s(y) \geq L^m$. In order to prove the remaining bounds, we note that $\sum_{i=1}^N z_i^m$ is strictly convex in $z=[z_1, \ldots, z_N]$ for $m > 1$, and that the minimum of $\sum_{i=1}^N z_i^m$ over $z \in \mc S_N$ occurs at $z= \onebf/N$, and is equal to $N^{1-m}$. Similarly, for $m \in (0,1)$, $\sum_{i=1}^N z_i^m$ is strictly concave in $z$, and its maximum over $z \in \mc S_N$ occurs at $z= \onebf/N$, and is equal to $N^{1-m}$. Combining these facts with \eqref{eq:service-rate-normalized}, and noting that $y/L \in \mc S_N$, gives the lemma.
\qed
\end{proof}

\begin{lemma}[Service Rate Monotonicity Between Jumps]
\label{lem:service-rate-dynamics-between-jumps}
For any  $y \in \mc S^L_N$, $N \in \NN$, $L>0$, under the dynamics in \eqref{eq:inter-vehicle-distance-dynamics}, 
$$
\frac{d}{dt} s(y) \leq 0 \quad \text{      if } m > 1 \qquad \& \qquad \frac{d}{dt} s(y) \geq 0\quad \text{ if } m \in (0,1) \, ,
$$
where the equality holds true if and only if $y=\frac{L}{N}\onebf$.
\end{lemma}
\begin{proof}
The time derivative of service rate is given by:
\begin{align}
\frac{d}{dt} s(y) & = \frac{d}{dt} \sum_{i=1}^N y_i^m =  m \sum_{i=1}^N y_i^{m-1} \dot{y}_i \nonumber \\ 
\label{eq:service-rate-derivative}
& =  m \sum_{i=1}^N y_i^{m-1} \left( y_{i+1}^m - y_i^m\right) 
\end{align}
where the second equality follows by \eqref{eq:inter-vehicle-distance-dynamics}. The result then follows by application of Lemma~\ref{lem:appendix-general-summation}, and by noting that $g(z)=z^m$ is a strictly increasing function for all $m>0$, and $h(z)=z^{m-1}$ is strictly decreasing if $m \in (0,1)$, and strictly increasing if $m>1$.
\qed
\end{proof}

The following lemma quantifies the change in service rate due to departure of a vehicle. 

\begin{lemma}[Change in Service Rate at Departures]
\label{lem:service-rate-jumps}
Consider the departure of a vehicle that changes inter-vehicle distances from $y \in \mc S_N^L$ to $y^- \in \mc S_{N-1}^L$, for some $N \in \NN \setminus \{1\}$, $L > 0$. If $y_1 \geq 0$ and $y_2 \geq 0$ denote the inter-vehicle distances behind and in front of the departing vehicle respectively, at the moment of departure, then the change in service rate due to the departure satisfies the following bounds:
\begin{enumerate}
\item if $m > 1$, then $0 \leq s(y^-) - s(y) \leq (y_1+y_2)^m \left(1-2^{1-m} \right)$;
\item if $m \in (0,1)$, then $0 \leq s(y) - s(y^-) \leq \min\{y_1^m,y_2^m\}$.
\end{enumerate}
\end{lemma}
\begin{proof}
If $m>1$, then $\left(\frac{y_1}{y_1+y_2}\right)^m + \left(\frac{y_2}{y_1+y_2}\right)^m \leq \frac{y_1}{y_1+y_2} + \frac{y_2}{y_1+y_2} = 1$, i.e., $s(y^-)-s(y) = (y_1 + y_2)^m - y_1^m - y_2^m \geq 0$. One can similarly show that $s(y)-s(y^-) \geq 0$ if $m \in (0,1)$.

In order to show the upper bound on $s(y^-)-s(y)$ for $m>1$, we note that the minimum value of $z^m + (1-z)^m$ over $z \in [0,1]$ for $m > 1$ is $2^{1-m}$, and it occurs at $z=1/2$. Therefore, 
\begin{align*}
s(y^-)-s(y) = (y_1 + y_2)^m - y_1^m - y_2^m & = (y_1 + y_2)^m \left(1 - \left(\frac{y_1}{y_1+y_2}\right)^m - \left(\frac{y_2}{y_1+y_2}\right)^m \right) \\
& \leq (y_1+y_2)^m \left(1-2^{1-m} \right)
\end{align*}

The upper bound on $s(y)-s(y^-)$ for $m \in (0,1)$ can be proven as follows. Since $y_1^m \leq (y_1 + y_2)^m$, $s(y) - s(y^-) = y_1^m + y_2 ^m - (y_1 + y_2)^m \leq y_2^m$.  
Similarly, $s(y)-s(y^-) \leq y_1^m$. Combining, we get $s(y)-s(y^-) \leq \min \{y_1^m, y_2^m\}$. Note that, in proving this, we nowhere used the fact that $m \in (0,1)$. However, this bound is useful only for $m \in (0,1)$. 
\qed
\end{proof}

\begin{remark}[Change in Service Rate at Arrivals]
\label{rem:service-rate-jump-arrival}
The bounds derived in Lemma~\ref{lem:service-rate-jumps} can be trivially used to prove the following bounds for change in service rate at arrivals:
\begin{enumerate}
\item if $m > 1$, then $0 \leq s(y) - s(y^+) \leq (y_1+y_2)^m \left(1-2^{1-m} \right)$;
\item if $m \in (0,1)$, then $0 \leq s(y^+) - s(y) \leq \min\{y_1^m,y_2^m\}$,
\end{enumerate}
where $y_1$ and $y_2$ are the inter-vehicle distances behind and in front of the arriving vehicle respectively, at the moment of arrival.
\end{remark}


The following lemma will facilitate generalization of Lemma~\ref{lem:service-rate-dynamics-between-jumps}. In preparation for the lemma, let $f(y,m):=m \sum_{i=1}^N y_i^{m-1} \left( y_{i+1}^m - y_i^m\right)$ be the time derivative of service rate, as given in \eqref{eq:service-rate-derivative}.

\begin{lemma}
\label{lem:service-rate-dot-lower-bound}
For all $y \in \text{int}(\mc S_N^L)$, $N \in \NN \setminus \{1\}$, $L > 0$:
\begin{equation}
\label{eq:service-rate-der-m}
\frac{\partial}{\partial m} f(y,m)|_{m=1} = - L D\left(\frac{y}{L} || P^- \frac{y}{L}\right) \leq 0
\end{equation}
Additionally, if $L < e^{-2}$, then
\begin{equation}
\label{eq:service-rate-twice-der-m}
\frac{\partial^2}{\partial m^2} f(y,m)|_{m=1} \geq 0
\end{equation}
Moreover, equality holds true in \eqref{eq:service-rate-der-m} and \eqref{eq:service-rate-twice-der-m} if and only if $y = \frac{L}{N}\onebf$.
\end{lemma}
\begin{proof}
Taking the partial derivative of $f(y,m)$ with respect to $m$, we get that 
\begin{align*}
\frac{\partial}{\partial m} f(y,m) & = \frac{f(y,m)}{m} + m \sum_{i=1}^N \left( y_i^{m-1} y_{i+1}^m \left(\log y_i + \log y_{i+1} \right) - 2 y_i^{2m-1} \log y_{i} \right)
\end{align*}
In particular, for $m=1$: 
\begin{align*}
\frac{\partial}{\partial m} f(y,m) |_{m=1} & = f(y,1) + \sum_{i=1}^N \left(y_{i+1} \left(\log y_i + \log y_{i+1} \right) - 2 y_i \log y_i \right) \\
& = L \sum_{i=1}^N \frac{y_i}{L} \log \left(\frac{y_{i-1}/L}{y_i/L} \right) \\
& = - L D\left(\frac{y}{L} || P^- \frac{y}{L}\right)
\end{align*}
where, for the second equality, we used the trivial fact that $f(y,1)=0$. 
Taking second partial derivative of $f(y,m)$ w.r.t. $m$ gives:
\begin{align*}
\frac{\partial^2}{\partial m^2} f(y,m) = & \sum_{i=1}^N  y_i^{m-1}\log y_i  \left(y_{i+1}^m - y_i^m \right) + \sum_{i=1}^N y_i^{m-1} \left(y_{i+1}^m \log y_{i+1} - y_i^m \log y_i \right)\\
& + \sum_{i=1}^N \left( y_i^{m-1} y_{i+1}^m \left(\log y_i + \log y_{i+1} \right) - 2 y_i^{2m-1} \log y_{i} \right) \\
& + m \sum_{i=1}^N \left( y_i^{m-1} y_{i+1}^m \left(\log y_i + \log y_{i+1} \right)^2  - 4 y_i^{2m-1} \log^2 y_i \right)
\end{align*}
In particular, for $m=1$:
\begin{align}
\frac{\partial^2}{\partial m^2} f(y,m) |_{m=1} = & \sum_{i=1}^N \left(y_{i+1} - y_i \right) \log y_i + \sum_{i=1}^N \left(y_{i+1} \log y_{i+1} - y_i \log y_i \right) \nonumber \\
& + \sum_{i=1}^N \left(y_{i+1} \left(\log y_i + \log y_{i+1} \right) - 2 y_i \log y_i \right) \nonumber \\
& + \sum_{i=1}^N \left( y_{i+1} (\log y_i + \log y_{i+1})^2 - 4 y_i \log^2 y_i\right) \nonumber \\
\label{eq:second-derivative-m-final-eq}
= &  \sum_{i=1}^N \log^2 y_i \left(y_{i+1} - y_i \right) + 2 \sum_{i=1}^N \log y_i \left(y_{i+1} \log y_{i+1} + y_{i+1} - y_i \log y_i - y_i \right)  \\
\nonumber
\geq & \, 0
\end{align}
It is easy to check that, $\log z$, $\log^2 z$ and $z + z \log z$ are strictly increasing, strictly decreasing and strictly decreasing functions, respectively, for $z \in (0,e^{-2})$. Therefore, Lemma~\ref{lem:appendix-general-summation} implies that each of the two terms in \eqref{eq:second-derivative-m-final-eq} is non-negative, and hence the lemma.
\qed
\end{proof}

Lemma~\ref{lem:service-rate-dot-lower-bound} implies that, for sufficiently small $L$, $f(y,m)$ is locally convex in $m$. One can use this property along with an exact expression for $\frac{\partial}{\partial m}f(y,m)$ in Lemma~\ref{lem:service-rate-dot-lower-bound} at $m=1$, and the fact that $f(y,1)=0$ for all $y$, to develop a linear approximation in $m$ of $f(y,m)$ around $m=1$. The following lemma derives this approximation, as also suggested by Figure~\ref{fig:service-rate-dot-vs-m}.

\begin{figure}[htb!]
 \centering
 \includegraphics[width=5cm]{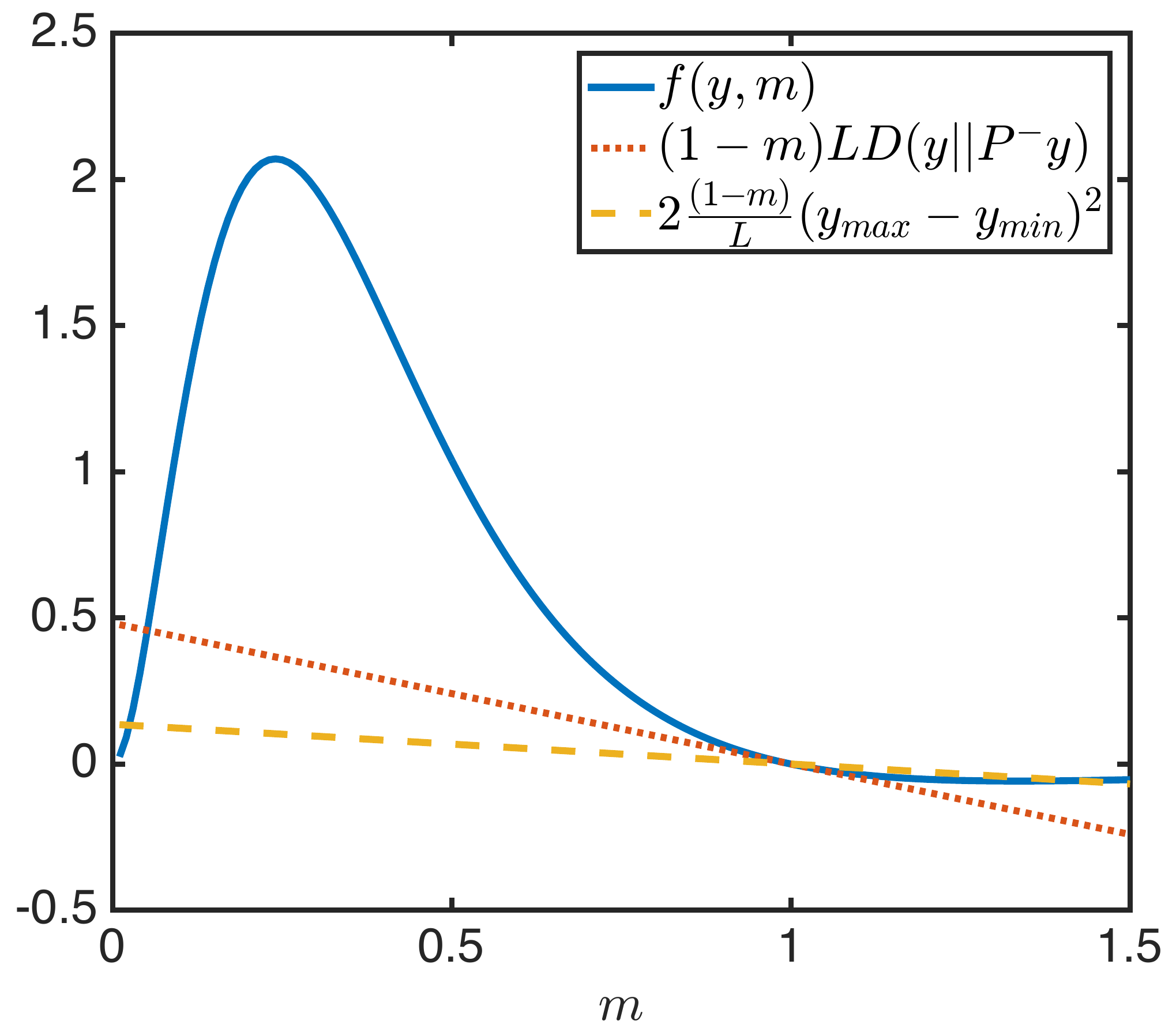} 
  \caption{$f(y,m)$ vs. $m$ for a typical $y \in \mc S_{10}$.}
    \label{fig:service-rate-dot-vs-m}
\end{figure}

\begin{lemma}
\label{lem:service-rate-dot-deltay-bound}
For a given $y \in \text{int}(\mc S_N^L)$, $n \in \NN$, $L \in (0,e^{-2})$, there exists $\underbar{m}(y) \in [0,1)$ such that 
$$
\frac{d}{dt} s(y) \geq 2 \frac{(1-m)}{L} \left(\ymax - \ymin \right)^2 \, , \qquad \forall \, m \in [\underbar{m}(y),1]
$$
\end{lemma}
\begin{proof}
For a given $y \in \text{int}(\mc S_N^L)$, the local convexity of $f(y,m):=\frac{d}{dt} s(y)$ in $m$, and the expression of $\frac{\partial}{\partial m}f(y,m)$ at $m=1$ in Lemma~\ref{lem:service-rate-dot-lower-bound}  implies that $\frac{d}{dt} s(y) \geq (1-m) L D\left(\frac{y}{L} || P^-\frac{y}{L} \right)$ for sufficiently small $m<1$. Pinsker's inequality implies $D\left(\frac{y}{L} || P^-\frac{y}{L} \right) \geq \frac{\|y-P^-y\|_1^2}{2 L^2}$. This, combined with the fact that $\|y-P^-y\|_1 \geq 2(\ymax-\ymin)$ for all $y \in \text{int}(\mc S_N^L)$, gives the lemma.
\qed
\end{proof}

%

\section{Busy Period Properties of the Horizontal Traffic Queue}
\label{sec:busy-period}

The system is called \emph{busy} when there is at least one vehicle on the road, or equivalently, the workload is positive. Once the system gets empty, it becomes \emph{idle} up to the time of next arrival. Thus, the system alternates between busy and idle periods. Accordingly, while the first busy period might start from a non-zero initial condition, if the first busy period terminates, then the subsequent busy periods will start from the zero initial condition. In this paper, unless otherwise stated explicitly, we shall implicitly assume a zero initial condition when referring to a busy period. 

\subsection{Expected Busy Period Duration}
The next lemma provides an expression for the expectation of the busy period duration in the linear case.

 \begin{lemma}\label{lemma:busy-period-mean}
 For any $\lambda < L/\bar \psi$, $L > 0$, $m=1$,  $\varphi \in \Phi$, $\psi \in \Psi$, the mean value of the busy period duration is equal to $\bar \psi / (L - \lambda \bar \psi)$.
 \end{lemma}
  \begin{proof}
A busy period, say of duration $B$, is initiated by the arrival of a vehicle, say $j$, when the system is idle. 
Let the number of vehicles that arrive during the busy period be $N_{bn}$. Note that $N_{bn}$ does not include the vehicle initiating the busy period. Therefore, the workload brought into the system during the busy period is equal to $w_B=\sum_{i=j}^{j+N_{bn}} d_i$.
The expected value of $N_{bn}$ can be obtained by conditioning on the duration of the busy period: 
\be\label{eq:nb-expectation}
E[N_{bn}]=E\left[E[N_{bn}|B]\right]=E[\lambda B]=\lambda E[B]
\ee
where the second equality follows from the fact that the arrival process is a Poisson process. 
Since the event $\{N_{bn}+1=n\}$ is independent of $\{d_{j+i},i > n\}$, $N_{bn}+1$ is a stopping time for the sequence $\{d_{j+i},i\geq 1\}$. Therefore, using Wald's equation, e.g., see \cite[Theorem 3.3.2]{Ross:96}, and \eqref{eq:nb-expectation}, the expected value of the workload $w_B$ added to the system during the busy period $B$ is given by: 
\be\label{eq:wb-expectation}
E[w_B]=(E[N_{bn}]+1) \, \bar \psi=(\lambda E[B]+1) \, \bar \psi.
\ee

\begin{figure}[htb!]
 \centering
 \includegraphics[width=5.5cm]{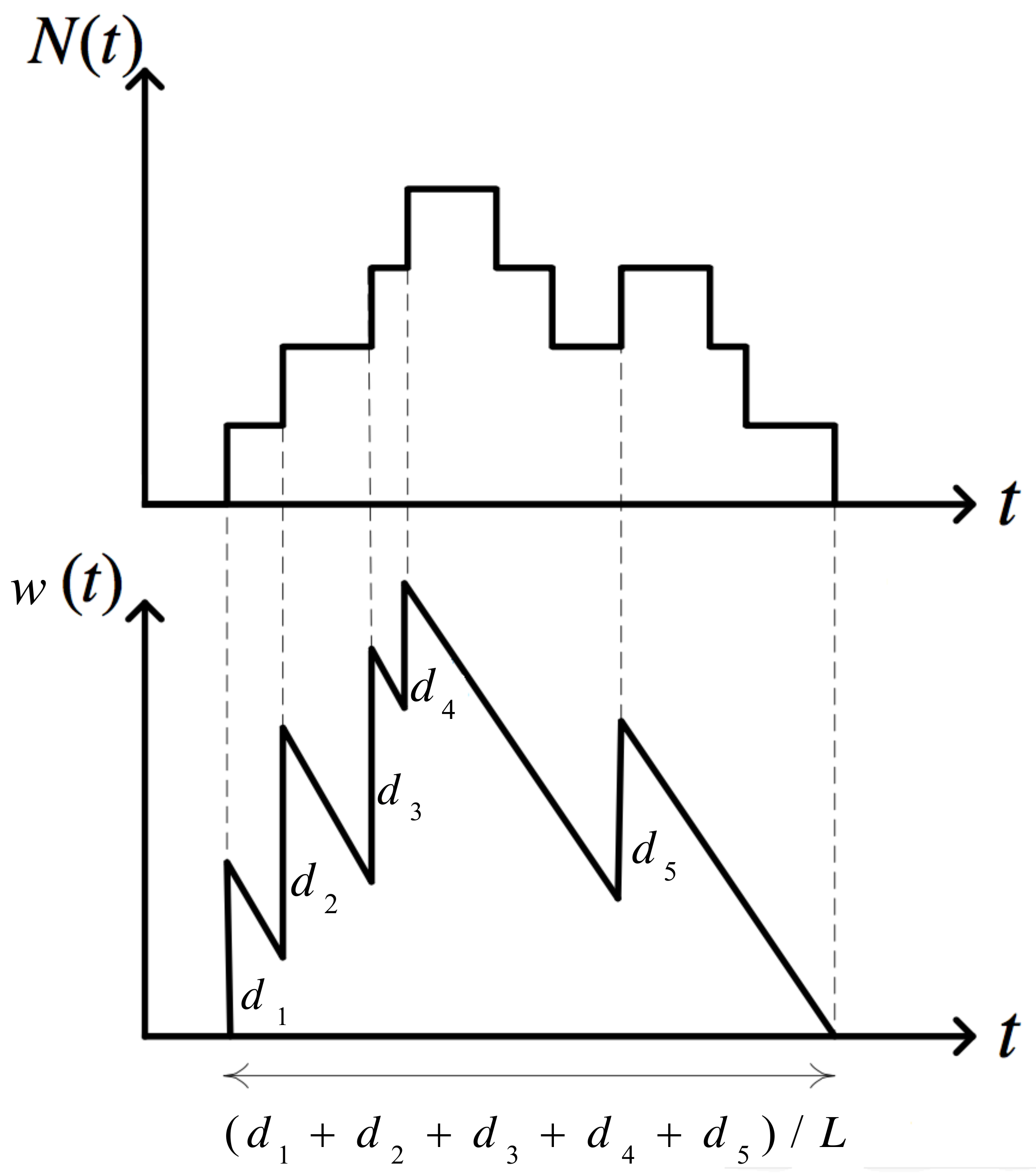} 
  \caption{(a) Queue length process and (b) workload process during a busy period. }
    \label{fig:busy-period}
\end{figure}

Since the workload decreases at a constant rate $L$ during a busy period, we have $B=w_B/L$ (see Figure \ref{fig:busy-period} for an illustration). Therefore, $E[B]=E[w_B]/L$, which when combined with \eqref{eq:wb-expectation}, establishes the lemma. 
\qed
\end{proof}

\begin{remark}
\label{remark:waiting-time}
Since the mean busy period duration is an upper bound on the mean waiting time, Lemma \ref{lemma:busy-period-mean} also gives an upper bound on the mean waiting time. One can then use Little's law~\cite{Kleinrock:75}\footnote{Little's law has previously been used in the context of processor sharing queues, e.g., in \cite{Altman.ea:06}.} to show that the mean queue length is upper bounded by $\lambda\bar \psi / (L - \lambda \bar \psi)$.
\end{remark}

Let $\mathcal{I}(t):=\int_0^t\delta_{\{w(s)=0\}}\de s$ be the cumulative \emph{idle time} up to time $t$. The following result characterizes the long run proportion of the idle time in the linear case.
\begin{proposition}
\label{prop:long-run-idle-time}
For any $\lambda<L/\bar \psi$, $m=1$, $L > 0$, $\varphi \in \Phi, \psi \in \Psi$, the long-run proportion of time in which HTQ is idle is given by the following:
$$\lim_{t\to\infty} \frac{\mathcal{I}(t)}{t}=1-\frac{\lambda \bar \psi}{L}>0 \quad a.s.$$
\end{proposition}

\begin{proof}
HTQ alternates between busy and idle periods. Let $Z=I+B$ be the duration of a cycle that contains an idle period of length $I$ followed by a busy period of length $B$. Idle period, $I$, has the same distribution as inter-arrival times i.e. an exponential random variable with mean $1/\lambda$, and the mean value of $B$ is given in Lemma \ref{lemma:busy-period-mean}. Note that duration of cycles, $Z$, are i.i.d. random variables. Thus, the busy-idle profile of the system is an alternating renewal process where renewals correspond to the moments at which the system gets idle.  Suppose the system earns reward at a rate of one per unit of time when it is idle (and thus the reward for a cycle equals the idle time of that cycle i.e. $I$). Then, the total reward earned up to time $t$ is equal to the total idle time in $[0,t]$ (or $\mathcal{I}(t)$), and by the result for renewal reward process (see \cite{Ross:96}, Theorem 3.6.1), with probability one,
$\lim_{t\to\infty} \mathcal{I}(t)/t=E[I]/(E[B]+E[I])$.  
\qed
\end{proof}

%

\subsection{Busy Period Distribution}
\label{sec:busy-period}
In this section, we compute the cumulative distribution function for the number of new arrivals during a busy period for a HTQ with constant service rate, say $p>0$. This could, e.g., correspond to \eqref{eq:inter-vehicle-distance-dynamics} for $m=1$. However, our analysis in this section, is not restricted to this specific model, but applies to any HTQ with constant service rate $p$. 
This cumulative distribution for the number of new arrivals during a busy period, while of independent interest, will be used to derive lower bounds on the throughput in the super-linear case in Section~\ref{subsec:superlinear}. Our analysis is inspired by that of M/G/1 queue, e.g., see \cite{Ross:96}, where our consideration for non-zero initial condition appears to be novel. 

Let us consider an arbitrary busy period spanning time interval $(0,t)$, without loss of generality. For non-zero initial condition, one has to distinguish between the first and subsequent busy periods. Let the workload at the beginning of the arbitrary busy period, denoted as $d_0$, be sampled from $\theta$. The relationship between $\theta$ and $\psi$ is as follows. 
If the system starts with a non-zero initial initial condition with initial workload $w_0>0$, then the value of the $d_0$ for the first busy period will be deterministic and equals $w_0$, and hence $\theta=\delta_{w_0}$. However, for subsequent busy periods, or if the initial condition is zero, $d_0$ is sampled from $\theta=\psi$. 
The workload brought to the system by arriving vehicles, $\{d_i\}_{i=1}^\infty$, equals to the distance that  vehicles wish to travel and are sampled identically and independently from the distribution $\psi$. When the system is busy, the workload decreases at a given constant rate $p > 0$. The busy period ends when the workload becomes zero. 
\begin{remark}
We emphasize that $d_0$ denotes the workload at the beginning of a busy period (see Figure \ref{fig:normalized-distance} for further illustration), and hence is not equal to zero when the queue starts from a zero initial condition. 
\end{remark}

In order to align our calculations with the standard M/G/1 framework, where service rate is assumed to be unity, we consider normalized workloads, $\tilde{d}_i:=d_i/p$ for all $i \in \{0,1,\cdots\}$ (see Figure \ref{fig:normalized-distance} for an illustration). Correspondingly, let the distributions for the normalized distances be denoted as $\tilde{\theta}$ and $\tilde{\psi}$. Let the arrival time of the $k$-th new vehicle during $(0,t)$ be denoted as $T_k$, and let $N_{bn}$ denote the number of arrivals in $(0,t)$, i.e., the total number of arrivals over the entire duration of the busy period, including the vehicle which initiates the busy period, is $N_{bn} + 1$. 

 A busy period ends at time $t$, and $N_{bn}=n-1$ if and only if, 
\begin{enumerate}[(i)]
\item $T_k\leq \tilde d_0+\cdots+ \tilde d_{k-1}, \qquad k=1,\cdots,n-1$ 
\item $\tilde d_0+\cdots+ \tilde d_{n-1}=t$
\item There are exactly $n-1$ arrivals in $(0,t)$
\end{enumerate} 

\begin{figure}[htbp]
\begin{center}
\includegraphics[width=8cm]{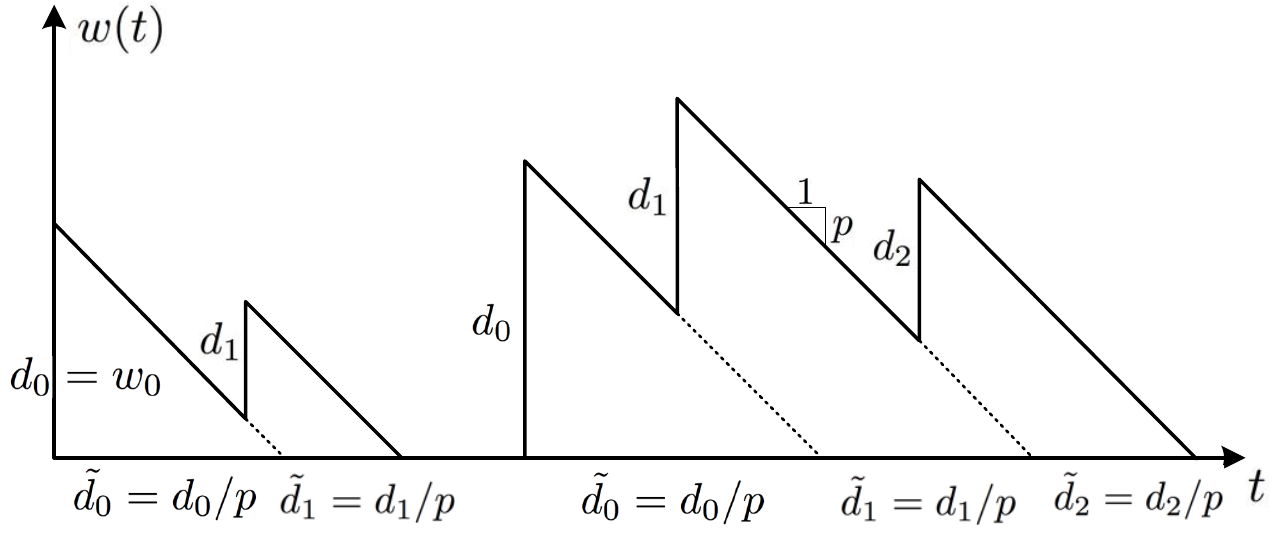} 
\caption{Evolution of workload during first two busy periods for an HTQ with constant service rate $p$, and starting from a non-zero initial condition. In the first busy period, $d_0$ is equal to the workload $w_0$ associated with the non-zero initial condition. In the second busy period, $d_0$ is equal to the workload brought by the first vehicle that initiates that busy period.}
\label{fig:normalized-distance}
\end{center}
\end{figure}

%

By treating densities as if they are probabilities, we get:
\begin{align}
\label{eq:prop-busy-des}
\Pr & (B =t \text{ and }N_{bn} = n -1)\nonumber \\
=&  \Pr(\tilde d_0+\cdots+ \tilde d_{n-1}=t, n-1 \text{ arrivals in $(0,t)$}, T_k\leq \tilde d_0+\cdots+ \tilde d_{k-1}, k=1,\cdots,n-1) \nonumber \\
=& \int_0^t\Pr(T_k\leq \tilde d_0+\cdots+ \tilde d_{k-1}, k=1,\cdots,n-1| n-1 \text{ arrivals in $(0,t)$}, \tilde d_0+\cdots+ \tilde d_{n-1}=t, \tilde d_0 =z) \nonumber \\
& \times \Pr(n-1 \text{ arrivals in $(0,t)$}, \tilde d_1+\cdots+ \tilde d_{n-1}=t-z)\, \tilde \theta(z)  \, \de z 
\end{align}
where we recall that $B$ is the random variable corresponding to the busy period duration. By the independence of normalized distances and the arrival process, the second probability term in the integrand in \eqref{eq:prop-busy-des} can be expressed as
\begin{align}
\label{eq:prop-busy-part2}
\Pr(n-1 \text{ arrivals in $(0,t)$}, \tilde d_1+\cdots+ \tilde d_{n-1}=t-z)=e^{-\lambda t}\frac{(\lambda t)^{n-1}}{(n-1)!} \tilde \psi_{n-1}(t- z) 
\end{align}
where $\tilde\psi_n$ is the $n$-fold convolution of $\tilde\psi$ with itself. 

In the first probability term in \eqref{eq:prop-busy-des}, it is given that the system receives $n-1$ arrivals in $(0,t)$ and since the arrival process is a Poisson process, the ordered arrival times, $\{T_1,T_2,\cdots,T_{n-1}\}$, are distributed as the ordered values of a set of $n-1$ independent uniform $(0,t)$ random variables $\{a_1,a_2,\cdots,a_{n-1}\}$ (see Theorem 2.3.1 in \cite{Ross:96}). Thus,
\begin{align}\label{eq:busy-period-given-n-prob}
\Pr & (T_k\leq \tilde d_0+\cdots+ \tilde d_{k-1}, k=1,\cdots,n-1| n-1 \text{ arrivals in $(0,t)$}, \tilde d_0+\cdots+ \tilde d_{n-1}=t,\tilde d_0 =z) \nonumber \\
&  =  \Pr(a_k\leq \tilde d_0+\cdots+\tilde d_{k-1}, k=1,\cdots,n-1|\tilde d_0+\cdots+\tilde d_{n-1} =t,\tilde d_0 =z)
\end{align}


%


By noting that $t-U$ will also be a uniform $(0,t)$ random variable whenever $U$ is, it follows that $a_1,\cdots,a_{n-1}$ has the same joint distribution as $t-a_{n-1},\cdots,t-a_1$. Thus, replacing $a_k$ with $a_{n-k}$ for $k\in\{1,\cdots,n-1\}$ in \eqref{eq:busy-period-given-n-prob}, we get
\begin{align}\label{eq:prob-manipulation}
 \Pr&(a_k  \leq \tilde d_0+\cdots+\tilde d_{k-1}, k=1,\cdots,n-1|\tilde d_0+\cdots+\tilde d_{n-1} =t,\tilde d_0 =z) \nonumber \\
 &  =\Pr(t-a_{n-k}  \leq \tilde d_0+\cdots+\tilde d_{k-1}, k=1,\cdots,n-1|\tilde d_0+\cdots+\tilde d_{n-1} =t,\tilde d_{0} =z)  \nonumber \\
 & = \Pr(t-a_{n-k}  \leq  t-(\tilde d_{k}+\cdots+\tilde d_{n-1}), k=1,\cdots,n-1|\tilde d_0+\cdots+\tilde d_{n-1} =t,\tilde d_0 =z)  \nonumber \\
  & = \Pr(a_{n-k}  \geq  \tilde d_{k}+\cdots+\tilde d_{n-1}, k=1,\cdots,n-1|\tilde d_0+\cdots+\tilde d_{n-1} =t,\tilde d_0 =z)  \nonumber \\
  & = \Pr(a_{n-k}  \geq  \tilde d_{k}+\cdots+\tilde d_{n-1}, k=1,\cdots,n-1|\tilde d_0+\cdots+\tilde d_{n-1} =t,\tilde d_0 =z) = \begin{cases}  z/t &  z<t \\
                          0      & \text{otherwise}
    \end{cases}
\end{align}
where the last equality follows from Lemma \ref{lemma:busy-period-type1-prob}.
If we let $H_p(t,n,\theta):=\Pr\{B\leq t, N_{bn}=n-1\}$ when the service rate equals $p$, and $d_0$ has distribution $\theta$; then, by plugging \eqref{eq:prop-busy-part2} and \eqref{eq:prob-manipulation} in \eqref{eq:prop-busy-des}, we get
\begin{equation*}
\frac{d}{dt}H_p(t,n,\theta) = e^{-\lambda t}\frac{(\lambda t)^{n-1}}{t(n-1)!}\int_0^t z \tilde\psi_{n-1}(t-z)\tilde \theta(z)\de z
\end{equation*}

By recalling the two special cases of interest to us: $\theta=\delta_{w_0}$ for a given non-zero initial workload $w_0$, and $\theta=\psi$ for zero initial condition, and using Lemma \ref{lemma:t-n},  we get that 
\begin{equation}
\label{eq:G-def}
G_p(t,n,\theta):=\frac{d}{dt}H_p(t,n,\theta)=\begin{cases} e^{-\lambda t}\frac{(\lambda t)^{n-1}w_0}{t(n-1)!p}\tilde \psi_{n-1}(t-w_0/p) & \theta = \delta_{w_0} \\ e^{-\lambda t}\frac{(\lambda t)^{n-1}}{n!}\tilde \psi_n(t) & \theta = \psi \end{cases}
\end{equation}

For $r \in \NN$, let $G_{r,p}(t,n,\theta)$ be the $r$-fold convolution of $G_p(t,n,\theta)$, defined in \eqref{eq:G-def}, with respect to $t$. In words, $G_{r,p}(t,n,\theta)$ is the probability that the number of new arrivals in each of (any) $r$ busy periods is equal to $n-1$, and that the sum of durations of all the busy periods is equal to $t$. Similarly, for non-zero initial condition, let $G_{p_1}(\theta_1) * G_{r-1,p_2}(\theta_2) (t,n)$ be the probability that the number of new arrivals in each of (any) $r$ busy periods is equal to $n-1$, and that the sum of durations of all the busy periods is equal to $t$, when the constant service rate for the first busy period is $p_1$ and is $p_2$ for the rest of the $r-1$ busy periods.

\section{Throughput Analysis}
\label{sec:throughput-analysis}

\subsection{Linear Case: $m=1$}
\label{sec:linear}
In this section, we provide an exact characterization of throughput for the linear case, i.e., when $m=1$. Recall that, for $m=1$, the service rate $s(y)=\sum_{i=1}^N y_i \equiv L$ is constant.

\begin{proposition}\label{prop:unstable}
For any $L > 0$, $\varphi \in \Phi$, $\psi \in \Psi$, $x_0 \in [0,L]^{n_0}$, $n_0 \in \NN$ and :
$$
\lambdamax(L,m=1,\varphi,\psi,x_0,\delta=0) \leq L/\bar{\psi} \, .
$$ 
\end{proposition}
\begin{proof}
By contradiction, assume $\lambdamax>L/\bar{\psi}$. Let $r(t):=\sum_{i=1}^{A(t)} d_i$ be the workload added to the system by the $A(t)$ vehicles that arrive over $[0,t]$. Therefore, 
\be\label{eq:workload-linear}
w(t)=w_0+r(t)-L(t-\mathcal{I}(t))
\ee
where $w_0$ is the initial workload. The process $\{r(t), \, t\geq 0\}$ is a renewal reward process, where the renewals correspond to arrivals of vehicles and the rewards correspond to the distances $\{d_i\}_{i=1}^{\infty}$ that vehicles wish to travel in the system upon arrival before their departures. 
Inter-arrival times are exponential random variables with mean $1/\lambda$, and the reward associated with each renewal is independently and identically sampled from $\psi$, whose mean is $\bar{\psi}$.  
%
%
Therefore, e.g., \cite[Theorem 3.6.1]{Ross:96} implies that, with probability one,  
\be
\label{eq:slln-renewal}
\lim_{t\to\infty}\frac{r(t)}{t}=\lambda \bar{\psi}
\ee
Thus, for all $\varepsilon \in \left(0,\lambda \bar{\psi} -L\right)$, there exists a $t_0\geq0$ such that, with probability one,
\be\label{eq:s-t-lowerbound}
\frac{r(t)}{t}\geq\lambda \bar \psi-\varepsilon/2> L + \varepsilon/2 \qquad \forall \, t \geq t_0.
\ee
Since $w_0$ and $\mathcal{I}(t)$ are both non-negative, \eqref{eq:workload-linear} implies that $w(t)\geq r(t) - Lt$ for all $t \geq 0$. 
%
This combined with \eqref{eq:s-t-lowerbound} implies that, with probability one, $w(t) \geq \varepsilon t /2$ for all $t \geq t_0$, and hence 
$\lim_{t\to\infty}w(t)= + \infty$. This combined with \eqref{eq:workload-upperbound} implies that, with probability one, $\lim_{t\to\infty}N(t)=+\infty$. 
\qed
\end{proof}

\begin{theorem}\label{thm:stable}
For any $L > 0$, $\varphi \in \Phi$, $\psi \in \Psi$, $x_0 \in [0,L]^{n_0}$, $n_0 \in \NN$:
$$
\lambdamax(L,m=1,\varphi,\psi,x_0,\delta=1) = L/\bar{\psi} \, .
$$ 
\end{theorem}
\begin{proof}
Assume that for some $\lambda < L/\bar{\psi}$, there exists some initial condition $(x_0,n_0)$ such that the queue length grows unbounded with some positive probability. 
Since the workload brought by every vehicle is i.i.d., and the inter-arrival times are exponential, without loss of generality, we can assume that the queue length never becomes zero. That is, the idle time satisfies $\mathcal{I}(t) \equiv 0$. Moreover, \eqref{eq:slln-renewal} implies that, for every $\varepsilon \in \left(0, L- \lambda \bar{\psi} \right)$, there exists $t_0 \geq 0$ such that, with probability one,
\begin{equation}
\label{eq:s-t-upperbound}
\frac{r(t)}{t} \leq \lambda \bar{\psi} + \varepsilon/2 < L - \varepsilon/2 \qquad \forall t \geq t_0
\end{equation} 

Combining \eqref{eq:workload-linear} with \eqref{eq:s-t-upperbound}, and substituting $\mathcal{I}(t) \equiv 0$, we get $w(t) < w_0 - \varepsilon t /2$, which implies that workload, and hence queue length, goes to zero in finite time after $t_0$, leading to a contradiction.  
%
Combining this with the upper bound proven in Proposition~\ref{prop:unstable} gives the result. 
%
%
%
\qed
\end{proof}

\begin{remark}
Theorem~\ref{thm:stable} implies that the throughput in the linear case is equal to the inverse of the 
time required to travel average total distance by a solitary vehicle in the system. In the linear case, the throughput can be characterized with probability one, independent of the initial condition of the queue. 
\end{remark}

\subsection{Monotonicity of Throughput in $m$ and $x_0$}
\label{sec:non-linear}


In this section, we show the following monotonicity property of $\lambdamax$ with respect to $m$ for small values of $L$: for given $x_0 \in [0,L]^{n_0}$, $n_0 \in \NN$, $L \in (0,1)$, $\varphi \in \Phi$, and $\psi \in \Psi$, throughput is a monotonically decreasing function of $m$. 
%
For this section, we rewrite \eqref{eq:dynamics-moving-coordinates} in $\RR_+^N$, i.e., without projecting onto $[0,L]^N$. Specifically, let the vehicle coordinates be given by the solution of 
\begin{equation}
\label{eq:x-dynamics-in-Rn}
\dot{x}_i = y_i^m, \qquad x_i(0)=x_{0,i}, \qquad i \in \until{N} 
\end{equation} 
Let $X(t;x_0,m)$ denote the solution to \eqref{eq:x-dynamics-in-Rn} at $t$ starting from $x_0$ at $t=0$. We will compare $X(t;x_0,m)$ under different values of $m$ and initial conditions $x_0$, over an interval of the kind $[0,\tau)$, in between arrivals and departures. 
We recall the notation that, if $x_{0}^1$ and $x_{0}^2$ are vectors of different sizes, then $x_{0}^1 \leq x_{0}^2$ implies element-wise inequality only for components which are common to $x_{0}^1$ and $x_{0}^2$. In Lemma~\ref{lemma:two-size-compare} and Proposition~\ref{prop:queue-length-2m}, this common set of components corresponds to the set of vehicles common between $x_{0}^1$ and $x_{0}^2$.
 
%

\begin{lemma}
\label{lemma:two-size-compare}
For any $L \in (0,1]$, $x^1_{0} \in \RR_+^{n_1}$, $x^2_{0} \in \RR_+^{n_2}$, $n_1, n_2 \in \NN$,
$$
x^1_{0} \leq x^2_{0}, \, \, n_2 \leq n_1, \, \,  0 < m_2 \leq m_1 \implies X(t;x^1_{0},m_1) \leq X(t;x^2_{0},m_2) \qquad \forall \, t \in [0,\tau)
$$ 
\end{lemma}
%
 \begin{proof}
The proof is straightforward when $n_1=n_2$. This is because, in this case, since $y_i \leq L\leq1$,  $m_2 \leq m_1$ implies $y_i^{m_2} \geq y_i^{m_1}$ for all $i \in \until{n_1}$. Using this with Lemmas~\ref{lem:monotonicity-same-size} and \ref{lem:car-following-type-K} gives the result. 
 
In order to prove  the result for $n_2 < n_1$, we show that $X(t; x^1_{0}, m_1) \leq X(t; x^2_{0}, m_1) \leq X(t; x^2_{0}, m_2)$. Note that the second inequality follows from the previous case. Therefore, it remains to prove the first inequality. Let $(i_1, \ldots, i_{n_2})$ be the set of indices of $n_2$ vehicles such that $0 \leq x^2_{0,i_1} \leq \ldots \leq x^2_{0,i_{n_2}} \leq L$. Similarly, let $(i_1, i_1+1, \ldots, i_2, i_2+1, \ldots)$ be the indices of $n_1$ vehicles in the order of increasing coordinates in $x_0^1$. Our assumption on the initial condition implies that $x^1_{0,i_k} \leq x^2_{0,i_k}$ for all $k \in \until{n_2}$. For brevity, let $x^1(t) \equiv X(t;x_0^1,m_1)$, and $x^2(t) \equiv X(t;x_0^2,m_1)$. It is easy to check that, for all $t \in [0,\tau)$, and all $k \in \until{n_2}$, 
\begin{equation}
\label{eq:xdot-ub}
\dot{x}^1_{i_k} = \left(x^1_{i_{k}+1} - x^1_{i_k}\right)^{m_1} \leq \left(x^1_{i_{k+1}} - x^1_{i_k}\right)^{m_1}
\end{equation}
Let $t \in [0,\tau)$ be the first time instant when $x^1_{i_k}(t)=x^2_{i_k}(t)$ for some $k \in \until{n_2}$. Then, recalling $x^1_{i_{k+1}}(t) \leq x^2_{i_{k+1}}(t)$, \eqref{eq:xdot-ub} implies that $\dot{x}_{i_k}^1(t) \leq \left(x^2_{i_{k+1}} - x^2_{i_k} \right)^{m_1} = \dot{x}^2_{i_k}(t)$. The result then follows from Lemma~\ref{lem:monotonicity-same-size}.  
%
%
\qed
 \end{proof}
 
 

Lemma~\ref{lemma:two-size-compare} is used to establish monotonicity of throughput as follows.
\begin{proposition}
\label{prop:queue-length-2m}
For any $L \in (0,1]$, $\varphi \in \Phi$, $\psi \in \Psi$, $\delta \in (0,1)$, $x^1_{0} \in [0,L]^{n_1}$, $x^2_{0} \in [0,L]^{n_2}$, $n_1, n_2 \in \NN$:
$$
x^1_{0} \leq x^2_{0}, \, \, n_2 \leq n_1, \, \,  0 < m_2 \leq m_1 \implies \lambdamax(L,m_1,\varphi,\psi,x^1_0,\delta)  \leq \lambdamax(L,m_2,\varphi,\psi,x^2_0,\delta)
$$
\end{proposition}
\begin{proof}
For brevity in notation, we refer to the queue corresponding to $m_1$, and initial condition $x_0^1$ as HTQ-S. We refer to the other queue as HTQ-F. Let $\lambda$, $\varphi$ and $\psi$ common to HTQ-S and HTQ-F be given.  Let $x^1(t) \equiv X(t; x^1_0, m_1)$ and $x^2(t) \equiv X(t; x^2_0, m_2)$, and let $N_s(t)$ and $N_f(t)$ be the queue lengths in the two queues at time $t$. It suffices to show that $N_s(t) \geq N_f(t)$ for a given realization of arrival times, arrival locations, and travel distances. In particular, this also implies that the departure locations are also the same for every vehicle, including the vehicles present at $t=0$, in both the queues.  
%
%

Indeed, it is sufficient to show that $x^1(\tau) \leq x^2(\tau)$ and $N_s(\tau)\geq N_f(\tau)$ where $\tau$ is the time of first arrival or departure from either HTQ-S or HTQ-F. Accordingly, we consider two cases, corresponding to whether $\tau$ corresponds to arrival or departure. 

Since $x^1(t) \leq x^2(t)$ for all $t \in [0,\tau)$ from Lemma~\ref{lemma:two-size-compare}, and the departure locations of all the vehicles in HTQ-S and HTQ-F are identical, the first departure from HTQ-S can not happen before the first departure in HTQ-F. Therefore, $N_s(\tau) \geq N_f(\tau)$. Since $x^1(\tau^-) \leq x^2(\tau^-)$, and $x^2(\tau)$ is a subset of $x^2(\tau^-)$, we also have $x^1(\tau) \leq x^2(\tau)$. 

When $\tau$ corresponds to the time of the first arrival, since the arrivals happen at the same location in HTQ-S and HTQ-F, and since $x^1(\tau^-) \leq x^2(\tau^-)$, rearrangement of the indices of the vehicles to include the new arrival at $t=\tau$ implies that $x^1(\tau) \leq x^2(\tau)$. Moreover, since $N_s(\tau^-) \geq N_f(\tau^-)$, and the arrivals happen simultaneously in both HTQ-S and HTQ-F, we have $N_s(\tau) \leq N_f(\tau)$. 
\qed
\end{proof}

\begin{remark}
Proposition \ref{prop:queue-length-2m} establishes monotonicity of throughput only for $L\in(0,1]$. This is consistent with our simulation studies, e.g., as reported in Figure \ref{fig:throughput-simulations}, according to which, the throughput is non-monotonic for large $L$. 
%
\end{remark}
For the analysis of the linear car following model, we exploited the fact that the total service rate of the system is constant. However, for the nonlinear model, i.e., $m \neq 1$, the total service rate depends on the number and relative locations of vehicles. The state dependent service rate of nonlinear models makes the throughput analysis much more complex. In the next section, we find probabilistic bound on the throughput in the super-linear case.
%






\subsection{Throughput Bounds for the Super-linear Case from Busy Period Calculations}\label{subsec:superlinear}
In this section, we derive lower bound on the throughput for the super-linear case. The next result computes a bound on the probability that the queue length of the HTQ satisfies a given upper bound over a given time interval, using the probability distribution functions from \eqref{eq:G-def}.
In Propositions~\ref{prop:nbp-lower-bound} and \ref{prop:nbn-lower-bound}, for the sake of clarity, we add explicit dependence on $\lambda$ to this probability distribution function.

\begin{proposition}\label{prop:nbp-lower-bound}
For any $m>1$, $M\in \mathbb{N}$, $L>0$, $\lambda >0$, $\varphi\in \Phi$, $\psi \in \Psi$, and zero initial condition $x_0=0$, the probability that the queue length is upper bounded by $M$ over a given time interval $[0,T]$ satisfies the following bound: 
\begin{equation}
\label{eq:nbp-lower-bound}
  \Pr \big(N(t) \leq M \quad \forall t \in [0,T] \big)\geq \sup_{r \in \NN} \, \sum_{n=1}^M \int_T^{\infty} G_{r,L^m M^{1-m}}(t,n,\psi,\lambda) \, \de t
  \end{equation}
 \end{proposition}
 \begin{proof}
Let us denote the current queueing system as HTQ-f. We shall compare queue lengths between HTQ-f and a slower queueing system HTQ-s, which starts from the same (zero) initial condition, and experiences the same realizations of arrival times, locations and travel distances. Let every incoming vehicle into HTQ-s and HTQ-f be tagged with a unique identifier. At time $t$, let $\mc J(t)$ be the set of identifiers of vehicles present both in HTQ-s and HTQ-f, $\mc J_{s/f}(t)$ be the set of identifiers of vehicles present only in HTQ-s, and $\mc J_{f/s}(t)$ be the set of identifiers of vehicles present only in HTQ-f.
Let $v^f_i$ denote the speed of the vehicle in HTQ-f with identifier $i \in \mc J(t) \cup \mc J_{f/s}(t)$, as determined by the car-following behavior underlying \eqref{eq:dynamics-moving-coordinates}. The vehicle speeds in HTQ-s are not governed by the car following behavior, but are rather related to the speeds of vehicles in HTQ-f as: 
 \begin{equation}
 \label{eq:htq-s-speed}
		v^s_i(t) =\left\{\begin{array}{ll}
		\ds v^f_i(t) \frac{p}{v^f(t)} \frac{|\mc J(t)|}{|\mc J(t)| + |\mc J_{s/f}(t)|} 	&i \in \mc J(t)\\[15pt]
		\ds \frac{p}{|\mc J(t)| + |\mc J_{s/f}(t)|}& i \in \mc J_{s/f}(t)\,\end{array}\right.
	\end{equation}
%
where $v^f(t):=\sum_{i \in \mc J(t)}v_i^f(t)$ is the sum of speeds of vehicles in HTQ-f that are also present in HTQ-s at time $t$, and $p$ is a parameter to be specified. Indeed, note that $\sum_{i \in \mc J(t) \cup \mc J_{s/f}(t)} v_i^s(t) \equiv p$, i.e., $p$ is the (constant) service rate of HTQ-s. 

Consider a realization where the number of arrivals into HTQ-s with $p=L^m M^{1-m}$ during any busy period overlapping with $[0,T]$ does not exceed $M$. We refer to such a realization as \emph{event} in the rest of the proof.  Since the maximum queue length during a busy period is trivially upper bounded by the number of arrivals during that busy period, conditioned on the event, we have 
\begin{equation}
\label{eq:slow-queue-queue-length-bound}
N_s(t) \leq M, \qquad t \in [0,T]
\end{equation}

Consider the union of departure epochs from HTQ-s and HTQ-f in $[0,T]$: $0=\tau_0 \leq \tau_1 \leq \ldots$. If $\mc J_{f/s}(\tau_k)=\emptyset$ for some $k \geq 0$, then $\mc J_{f/s}(t)=\emptyset$ for all $t \in (\tau_k,\tau_{k+1})$. Hence, the service rate for HTQ-f over the interval $(\tau_k,\tau_{k+1})$ is $v^f(t)$, which, conditioned on the event, is lower bounded by $L^m M^{1-m}=p$ by Lemma \ref{lem:service-rate-bounds}.
Therefore, $p/v^f(t) \leq 1$ over $(\tau_k,\tau_{k+1})$, and hence \eqref{eq:htq-s-speed} implies that all the vehicles with identifiers in $\mc J_f$ will travel slower in HTQ-s in comparison to HTQ-f. In particular, this implies that $\mc J_{f/s}(\tau_{k+1})=\emptyset$. This, combined with the fact that $\mc J_{f/s}(\tau_0)=\emptyset$ (both the queues start from the same initial condition), we get that, conditioned on the event, $\mc J_{s/f}(t) \equiv \emptyset$, and hence $N(t) \leq N_s(t)$ over $[0,T]$. Combining this with \eqref{eq:slow-queue-queue-length-bound} gives that, conditioned on the event, $N(t) \leq M$ over $[0,T]$.

We now compute the probability of the occurrence of the event using busy period calculations from Section~\ref{sec:busy-period}. The event can be categorized by the maximum number of busy periods, say $r \in \NN$, that overlap with $[0,T]$, i.e., the $r$-th busy period ends after time $T$ (and each of these busy periods has at most $M$ arrivals). Since these busy periods are interlaced with idle periods, the probability of the $r$-th busy period ending after time $T$ is lower bounded by the probability that the sum of the durations of $r$ busy periods is at least $T$. \eqref{eq:G-def} implies that the latter quantity is equal to $\sum_{n=1}^M \int_T^{\infty} G_{r,L^m M^{1-m}}(t,n,\psi,\lambda) \, \de t$. The proposition then follows by noting that this is true for any $r \in \NN$.
\qed
 \end{proof}
 
 \begin{remark}
In the proof of Proposition~\ref{prop:nbp-lower-bound}, when deriving probabilistic upper bound on the queue length over a given time horizon $[0,T]$, we neglected the idle periods in $[0,T]$. This introduces conservatism in the bound on the right hand side of \eqref{eq:nbp-lower-bound}. Since the idle period durations are distributed independently and identically according to an exponential random variable (since the arrival process is Poisson), one could incorporate them into \eqref{eq:nbp-lower-bound} by taking convolution of $G$ with idle period distributions. Our choice for not doing so here is to ensure conciseness in the presentation of bounds in \eqref{eq:nbp-lower-bound}. The resulting conservatism is also present in Proposition~\ref{prop:nbn-lower-bound}, and carries over to Theorems~\ref{thm:superlinear-bound-empty} and \ref{thm:superlinear-bound-non-empty}, as well as to the corresponding simulations reported in Figures~\ref{fig:lambda-m-empty}, \ref{fig:superlinear-large-L} and \ref{fig:phase-transition-non-empty}. 
%
%
\end{remark}
 
 The next result generalizes Proposition~\ref{prop:nbp-lower-bound} for non-zero initial condition. Note that the non-zero initial condition only affects the first busy period; all subsequent busy periods will necessarily start from with zero initial condition.

 \begin{proposition}\label{prop:nbn-lower-bound}
 For any $m>1$, $M\in\mathbb{N}$, $L>0$, $\lambda >0$, $\varphi\in \Phi$, $\psi \in \Psi$, initial condition $x_0 \in [0,L]^{n_0}$, $n_0 \in \NN$, with associated workload $w_0>0$, the probability that the queue length is upper bounded by $M+n_0$ over a given time interval $[0,T]$ satisfies the following: 
  $$\Pr \big(N(t) \leq M + n_0 \quad \forall t \in [0,T] \big)\geq \sup_{r \in \NN} \, \sum_{n=1}^M \int_T^{\infty} G_{L^m (M+n_0)^{1-m}}(\delta_{w_0}) * G_{r-1,L^m M^{1-m}}(\psi) (t,n,\lambda) \, \de t$$
 \end{proposition}
 \begin{proof}
The proof is similar to the proof of Proposition \ref{prop:nbp-lower-bound}; however, since we consider $M$ number of new arrivals in each of the busy periods, the \emph{event} of interest is when the queue length in HTQ-s does not exceed $M+n_0$ and $M$ in the first and subsequent busy periods, respectively, while operating with constant service rates $L^m (M+n_0)^{1-m}$ and $L^m M^{1-m}$, respectively. 
\qed
 \end{proof}

 We shall use Propositions~\ref{prop:nbp-lower-bound} and \ref{prop:nbn-lower-bound} to establish probabilistic lower bound for a finite time horizon version of the throughput defined in Definition \ref{def:throughput}: for $T>0$, let 
 \begin{equation*}
\label{eq:throughput-def-finite-horizon}
\lambdamax(L,m,\varphi,\psi,x_0,\delta,T):= \sup \left\{\lambda \geq 0: \, \Pr \left( N(t;L, m, \lambda, \varphi, \psi, x_0) < + \infty, \quad \forall t \in [0,T]  \right) \geq 1 - \delta  \right\}.
\end{equation*}
 
 

\begin{theorem}\label{thm:superlinear-bound-empty}
 For $L>0$, $m>1$, $\varphi\in\Phi$, $\psi\in\Psi$, $\delta\in(0,1)$, $T>0$, zero initial condition $x_0=0$, 
 \begin{equation}
 \label{eq:sublinear-throughput-zero-initial-condition}
 \lambda_{\max}(L, m, \varphi, \psi, x_0,\delta,T)\geq \sup_{M \in \NN}\;\sup \Big\{\lambda \geq 0 \; \Big | \sup_{r \in \NN} \, \sum_{n=1}^M \int_T^{\infty} G_{r,L^m M^{1-m}}(t,n,\psi,\lambda) \, \de t \geq1-\delta \Big\} 
 \end{equation}
 \end{theorem}
 \begin{proof}
Follows from Proposition \ref{prop:nbp-lower-bound}.
\qed
 \end{proof}

  \begin{theorem}\label{thm:superlinear-bound-non-empty}
 For $L>0$, $m>1$, $\varphi\in\Phi$, $\psi\in\Psi$, $\delta\in(0,1)$, $T>0$, initial condition $x_0 \in [0,L]^{n_0}$, $n_0 \in \NN$, with associated workload $w_0>0$, 
 \begin{multline*}
 \lambda_{\max}(L, m, \varphi, \psi, x_0,\delta,T) \\ \geq \sup_{M \in \NN}\;\sup \Big \{\lambda>0 \; \Big | \sup_{r \in \NN} \, \sum_{n=1}^M \int_T^{\infty} G_{L^m (M+n_0)^{1-m}}(\delta_{w_0}) * G_{r-1,L^m M^{1-m}}(\psi) (t,n,\lambda)\geq 1-\delta \Big\} 
 \end{multline*}
 \end{theorem}
 \begin{proof}
 Follows from Proposition \ref{prop:nbn-lower-bound}.
 \qed
 \end{proof}
 
 \begin{remark}
 In Theorems~\ref{thm:superlinear-bound-empty} and \ref{thm:superlinear-bound-non-empty}, we implicitly assume the rather standard convention that supremum over an empty set is zero.
 \end{remark}

\subsection{Throughput Bounds under Batch Release Control Policy}
In this section, we consider a \emph{time-perturbed} version  of the arrival process. 
For a given realization of arrival times, $\{t_1,t_2,\cdots\}$, consider a perturbation map $t_i^{\prime} \equiv t_i^{\prime}(t_1,\ldots,t_i)$ satisfying $t_i^{\prime}\geq t_i$ for all $i$, which prescribes the perturbed arrival times. The magnitude of perturbation is defined as $\eta := E\left(t_i^{\prime}-t_i\right)$, where the expectation is with respect to the Poisson process with rate $\lambda$ that generates the arrival times.

We prove boundedness of the queue length under a specific perturbation map. This perturbation map is best understood in terms of a control policy that governs the release of arrived vehicles into HTQ. In order to clarify the implementation of the control policy, we decompose the proposed HTQ into two queues in series: denoted as HTQ1 and HTQ2, both of which have the same geometric characteristics as HTQ, i.e., a circular road segment of length $L$ (see Figure \ref{fig:htq1-htq2} for illustrations). The original arrival process for HTQ, i.e. spatio-temporal Poisson process with rate $\lambda$ and spatial distribution $\varphi$ is now the arrival process for HTQ1. Vehicles remain stationary at their arrival locations in HTQ1, until released by the control policy into HTQ2. Upon released into HTQ2, vehicles travel according to \eqref{eq:dynamics-moving-coordinates} until they depart after traveling a distance that is sampled from $\psi$, as in the case of HTQ. The time of release of the vehicles into HTQ2 correspond to their perturbed arrival times $t_1^{\prime}, t_2^{\prime}, \ldots$. The average waiting time in HTQ1 under the given release control policy is then the magnitude of perturbation in the arrival times.

%

\begin{figure}[htb!]
\begin{center}
\includegraphics[width=14cm]{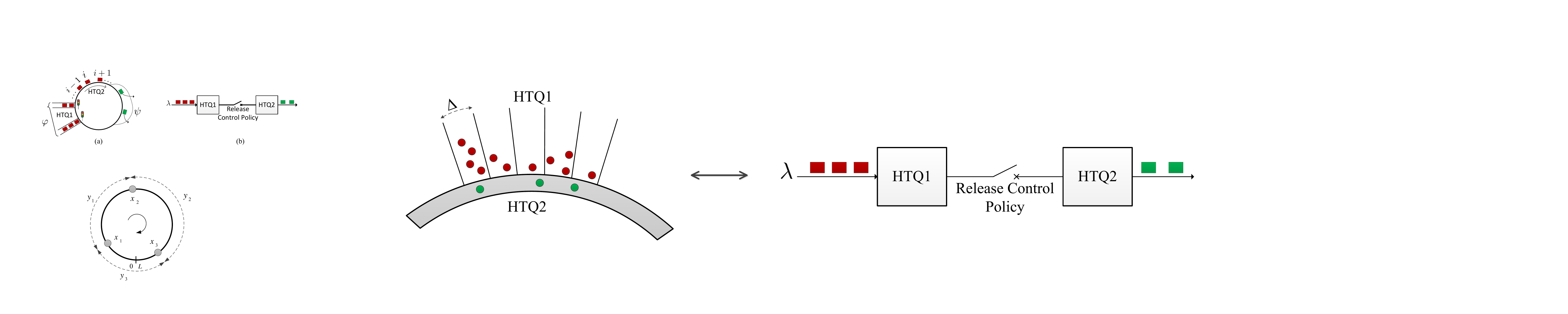} 
\caption{Decomposition of HTQ into HTQ1 and HTQ2 in series.}
\label{fig:htq1-htq2}
\end{center}
\end{figure}

We consider the following class of release control policy, for which we recall from the problem setup in Section~\ref{sec:problem-formulation} that $\text{supp}(\varphi)=[0,\ell]$ for some $\ell \in [0,L]$. 


\begin{definition}[Batch Release Control Policy $\pi^b_\triangle$]
\label{def:batch-release-control-policy}
Divide $[0,\ell]$ into sub-intervals, each of length $\triangle$, enumerated as $1, 2, \ldots, \lceil \frac{\ell}{\triangle} \rceil$. Let $T_1$ be the first time instant when HTQ2 is empty. At time $T_1$, release one vehicle each, if present, from all odd-numbered sub-intervals in $\{1, 2, \ldots, \lceil \frac{\ell}{\triangle} \rceil\}$ simultaneously into HTQ2. Let $T_2$ be the next time instant when HTQ2 is empty. At time $T_2$, release one vehicle each, if present, from all even-numbered sub-intervals in $\{1, 2, \ldots, \lceil \frac{\ell}{\triangle} \rceil\}$ simultaneously into HTQ2. Repeat this process of alternating between releasing one vehicle each from odd and even-numbered sub-intervals every time that HTQ2 is empty.
\end{definition}


\begin{remark}
\begin{enumerate}
\item Under $\pi^b_{\triangle}$, when vehicles are released into HTQ2, the inter-vehicle distances in the front and rear of each vehicle being released is at least equal to $\triangle$. 
\item The order in which vehicles are released into HTQ2 from HTQ1 under $\pi^b_{\triangle}$ may not be the same as the order of arrivals into HTQ1.
\end{enumerate}
\end{remark}

In the next two sub-sections, we analyze the performance of the batch release control policy for sub-linear and super-linear cases. 
\subsubsection{The Sub-linear Case}
In this section, we derive a lower bound on throughput when $m\in(0,1)$. We first derive a trivial lower bound in Proposition \ref{prop:throughput-bound-sublinear} implied by Lemma \ref{lem:service-rate-jumps} and Remark \ref{rem:service-rate-jump-arrival}. Next, we improve this lower bound in Theorem \ref{thm:main-sub-linear} under a under a batch release control policy, $\pi_{\triangle}^b$. 
\begin{proposition}
\label{prop:throughput-bound-sublinear}
For any $L > 0$, $m\in(0,1)$, $\varphi \in \Phi$, $\psi \in \Psi$, $x_0 \in [0,L]^{n_0}$, $n_0 \in \NN$:
$$
  \lambdamax(L,m,\varphi,\psi,x_0,\delta=0) \geq L^m/\bar{\psi} 
$$
\end{proposition}
\begin{proof}
Remark \ref{rem:service-rate-jump-arrival} implies that, for $m \in (0,1)$, the service rate does not decrease due to arrivals. Therefore, a simple lower bound on the service rate for any state is the service rate when there is only one vehicle in the system, i.e., $L^m$. Therefore,  
%
the workload process is upper bounded as
$w(t) = w_0+r(t)-\int_0^ts(z)\de z \leq w_0+r(t) -L^m(t-\mathcal{I}(t)), \quad \forall t\geq 0$, 
where $r(t)$ and $\mathcal{I}(t)$ denote the renewal reward and the idle time processes, respectively, as introduced in the proof of Proposition \ref{prop:unstable}. Similar to the proof of Proposition \ref{prop:unstable}, it can be shown that, if $\lambda < L^m/\bar{\psi}$, then the workload, and hence the queue length, goes to zero in finite time with probability one. 
%
\qed
\end{proof}

Next, we establish better throughput guarantees than Proposition~\ref{prop:throughput-bound-sublinear},  under a batch release control policy, $\pi_{\triangle}^b$.
The next result characterizes the time interval between release of successive batches into HTQ2 under $\pi_\triangle^b$.

\begin{lemma}
\label{eq:successive-release-time-difference}
For given $\lambda>0$, $\triangle > 0$, $\varphi \in \Phi$, $\psi \in \Psi$ with $\supp(\psi)=[0,R]$, $R>0$, $m\in(0,1)$, $x_0\in[0,L]^{n_0}$, $L>0$, $n_0\in \mathbb{N}$, let $T_1$, $T_2$, $\ldots$ denote the random variables corresponding to time of successive batch releases into HTQ2 under $\pi_{\triangle}^b$. Then, $T_1 \leq \frac{n_0R}{L^m}$, $T_{i+1}-T_i \leq R/\triangle^m$ for all $i \geq 1$, and $\ymin(t) \geq \triangle$ for all $t \geq T_1$.
\end{lemma}
\begin{proof}
Since the maximum distance to be traveled by every vehicle is upper bounded by $R$, the initial workload satisfies $w_0\leq n_0R$.  Since the minimum service rate for $m \in (0,1)$ is $L^m$ (see proof of Proposition~\ref{prop:throughput-bound-sublinear}), with no new arrivals, it takes at most $w_0/L^m=n_0 R/L^m$ amount of time for the system to become empty. This establishes the bound on $T_1$. 

Lemma~\ref{lem:vehicle-distance-monotonicity} implies that, under $\pi^b_{\triangle}$, the minimum inter-vehicle distance in HTQ2 is at least $\triangle$ after $T_1$. 
This implies that $\ymin(t) \geq \triangle$ for all $t \geq T_1$, and hence the minimum speed of every vehicle in HTQ2 is at least $\triangle^m$ after $T_1$. Since the maximum distance to be traveled by every vehicle is $R$, this implies that the time between release of a vehicle into HTQ2 and its departure is upper bounded by $R/\triangle^m$, which in turn is also an upper bound on the time required by all the vehicles released in one batch to depart from the system. 
\qed
\end{proof}
Let $N_1(t)$ and $N_2(t)$ denote the queue lengths in HTQ1 and HTQ2, respectively, at time $t$. Lemma~\ref{eq:successive-release-time-difference} implies that, for every $\triangle>0$, $N_2(t)$ is upper bounded for all $t \geq T_1$. The next result identifies conditions under which $N_1(t)$ is upper bounded. 




For $F>0$, let $\Phi_F:=\setdef{\varphi \in \Phi}{\sup_{x \in [0,\ell]} \varphi(x) \leq F}$. For subsequent analysis, we now derive an upper bound on the \emph{load factor}, i.e., the ratio of the arrival and departure rates, associated with a typical sub-queue of HTQ1 among  $\{1, 2, \ldots, \lceil \frac{\ell}{\triangle} \rceil\}$. It is easy to see that, for every $\varphi \in \Phi_F$, $F>0$, the arrival process into every sub-queue is Poisson with arrival rate upper bounded by $\lambda F \triangle$.  Lemma~\ref{eq:successive-release-time-difference} implies that the departure rate is at least $\triangle^m/2R$. Therefore, the load factor for every sub-queue is upper bounded as 
\begin{equation}
\label{eq:load-factor-upper-bound}
\rho \leq \frac{2 R \lambda F \triangle}{\triangle^m}=2 R \lambda F \triangle^{1-m}
\end{equation}
In particular, if 
\be\label{eq:triangle-star}
\triangle < \triangle^*(\lambda):=\left(2 R \lambda F \right)^{-\frac{1}{1-m}},
\ee 
then $\rho<1$. It should be noted that for $n_0<+\infty$, by Lemma \ref{eq:successive-release-time-difference},  $T_1<+\infty$. The service rate is zero during $[0,T_1]$; however, since $T_1$ is finite, this does not affect the computation of load factor. 


\begin{proposition}
\label{prop:pi1-policy}
For any $\lambda>0$, $\varphi \in \Phi_F$, $F>0$, $\psi \in \Psi$ with $\supp(\psi)=[0,R]$, $R>0$, $m\in(0,1)$, $x_0\in[0,L]^{n_0}$, $L>0$, $n_0\in \mathbb{N}$, for sufficiently small $\triangle$, $N_1(t)$ is bounded for all $t \geq 0$ under $\pi_{\triangle}^b$, almost surely.
\end{proposition}
\begin{proof}
By contradiction, assume that $N_1(t)$ grows unbounded. This implies that there exists at least one sub-queue, say $ j \in \{1, 2, \ldots, \lceil \frac{\ell}{\triangle} \rceil\}$, such that its queue length, say $N_{1,j}(t)$, grows unbounded. In particular, this implies that there exists $t_0 \geq T_1$ such that $N_{1,j}(t) \geq 2$ for all $t \geq t_0$. Therefore, for all $t \geq t_0$, the ratio of arrival rate to departure rate for the $j$-th sub-queue is given by \eqref{eq:load-factor-upper-bound}, which is a decreasing function of $\triangle$, and hence becomes strictly less than one for sufficiently small $\triangle$. A simple application of the law of large numbers then implies that, almost surely, $N_{1,j}(t)=0$ for some finite time, leading to a contradiction.
\qed 
\end{proof}

The following result gives an estimate of the mean waiting time in a typical sub-queue in HTQ1 under the $\pi_{\triangle}^b$ policy.

\begin{proposition}
\label{eq:waiting-time}
For $\varphi \in \Phi_F$, $F>0$, $\psi \in \Psi$, $m \in (0,1)$, there exists a sufficiently small $\triangle$ such that the average waiting time in HTQ1 under $\pi_{\triangle}^b$ is upper bounded as: 
\begin{equation}
\label{eq:W-upper-bound}
W \leq R (2 R \lambda F )^{\frac{m}{1-m}} \left(\frac{2}{m^{\frac{m}{1-m}}} + \frac{m}{m^{\frac{m}{1-m}}-m^{\frac{1}{1-m}}}  \right).
\end{equation}
\end{proposition}
\begin{proof}
It is easy to see that the desired waiting time corresponds to the system time of an M/D/1 queue with load factor given by \eqref{eq:load-factor-upper-bound} along with the arrival and departure rates leading to \eqref{eq:load-factor-upper-bound}. Note that, by Lemma \ref{eq:successive-release-time-difference}, for finite $n_0$, the value of $T_1$ is finite and does not affect the average waiting time. Therefore, using standard expressions for M/D/1 queue~\cite{Kleinrock:75}, we get that the waiting time in HTQ1 is upper bounded as follows for $\rho <1$:
\begin{align}
W \leq \frac{2R}{\triangle^m} + \frac{R}{\triangle^m} \frac{\rho}{1-\rho}  & \leq \frac{2R}{\triangle^m} + \frac{R}{\triangle^m} \frac{1}{1-\rho} \nonumber \\ 
\label{eq:waiting-time-upper-bound}
& \leq \frac{2R}{\triangle^m} + \frac{R}{\triangle^m -2 R \lambda F \triangle} 
\end{align} 
It is easy to check that the minimum of the second term in \eqref{eq:waiting-time-upper-bound} over $\big(0,\triangle^*(\lambda)\big)$ occurs at  $\triangle = \left(\frac{m}{2 R \lambda F} \right)^{\frac{1}{1-m}}$. Substitution in the right hand side of the first inequality in \eqref{eq:waiting-time-upper-bound} gives the result.
\qed
\end{proof}

\begin{remark}
\label{rem:W-limit}
\eqref{eq:W-upper-bound} implies that, for every $R>0$, $F>0$, $\lambda>0$, we have $W \to 2R$ as $m \to 0^+$. 
\end{remark}

We extend the notation introduced in \eqref{eq:throughput-def} to $\lambdamax(L,m,\varphi,\psi,x_0,\delta,\eta)$ to also show the dependence on maximum allowable perturbation $\eta$. This is not to be confused with the notation for $\lambdamax$ used in Theorems~\ref{thm:superlinear-bound-empty} and \ref{thm:superlinear-bound-non-empty}, where we used the notion of throughput over finite time horizons. We choose to use the same notations to maintain brevity. 

In order to state the next result, for given $R>0$, $F>0$, $m \in (0,1)$ and $\eta \geq 0$, let $\tilde{W}(m,F,R,\eta)$ 
be the value of $\lambda$ for which the right hand side of \eqref{eq:W-upper-bound} is equal to $\eta$, if such a $\lambda$ exists and is at least $L^m/\bar{\psi}$, and let it be equal to $L^m/\bar{\psi}$, otherwise. The lower bound of $L^m/\bar{\psi}$ in the definition of $\tilde{W}$ is inspired by Proposition~\ref{prop:throughput-bound-sublinear}. The next result formally states $\tilde{W}$ as a lower bound on $\lambdamax$.

\begin{theorem}\label{thm:main-sub-linear}
For any $\varphi \in \Phi_F$, $F>0$, $\psi \in \Psi$ with $\text{supp}(\psi)=[0,R]$, $R>0$, $m \in (0,1)$, $x_0 \in [0,L]^{n_0}$, $n_0 \in \NN$, $L>0$, and maximum permissible perturbation $\eta \geq 0$, 
$$
\lambdamax(L,m,\varphi,\psi,x_0,\delta=0,\eta) \geq \tilde{W}(m,F,R,\eta)
$$
In particular, if $\eta > 2R$, then $\lambdamax(L,m,\varphi,\psi,x_0,\delta=0,\eta) \to + \infty$ as $m \to 0^+$.
\end{theorem}
\begin{proof}
Consider any $\lambda \leq\tilde{W}(m,F,R,\eta)$, and $\triangle \leq \left(\frac{m}{2 R \lambda F} \right)^{\frac{1}{1-m}}$. Under policy $\pi_\triangle^b$,  Lemma \ref{eq:successive-release-time-difference} and Proposition \ref{prop:pi1-policy} imply that, for finite $n_0$, $N_2(t)$ and $N_1(t)$ remain bounded for all times, with probability one. Also, for $\lambda = \tilde{W}(m,F,R,\eta)$, by Proposition \ref{eq:waiting-time} and the definition of $\tilde{W}(m,F,R,\eta)$, the introduced perturbation remains upper bounded by $\eta$. Since the right hand side of \eqref{eq:W-upper-bound} is monotonically increasing in $\lambda$, perturbations remain bounded by $\eta$ for all $\lambda \leq \tilde{W}(m,F,R,\eta)$. 
In particular, by Remark \ref{rem:W-limit}, we have $W\to 2R$ as $m\to 0^+$. In other words, as $m\to 0^+$, the magnitude of the introduced perturbation becomes independent of $\lambda$. 
Therefore, when $\eta >2R$, and $m\to 0^+$  throughput can grow unbounded while perturbation and queue length remains bounded.
\qed
\end{proof}




\begin{remark}
We emphasize that the only feature required in a batch release control policy is that, at the moment of release, the front and rear distances for the vehicles being released should be greater than $\triangle$. The requirement of the policy in Definition~\ref{def:batch-release-control-policy} for the road to be empty at the moment of release makes the control policy conservative, and hence affects the maximum permissible perturbation. In fact, for special spatial distributions, e.g., when $\varphi$ is a Dirac delta function and the support of $\psi$ is $[0,L-\triangle])$, one can relax the conservatism to guarantee unbounded throughput for arbitrarily small permissible perturbation. 
\end{remark}

\subsubsection{The Super-linear Case}

In this section, we study the throughput for the super-linear case under perturbed arrival process with a maximum permissible perturbation of $\eta$. For this purpose, we consider the batch release control policy $\pi_{\triangle}^b$, defined in Definition \ref{def:batch-release-control-policy}, for our analysis. Time intervals  between release of successive batches, under $\pi_{\triangle}^b$, are characterized the same as Lemma \ref{eq:successive-release-time-difference}. However, in the super linear case, by Lemma \ref{lem:service-rate-bounds}, the initial minimum service rate is $L^mn_0^{1-m}$. Therefore, the time of first release is bounded as $T_1<n_0^mR/L^m$. Moreover, similar to the proof of Lemma \ref{eq:successive-release-time-difference}, it can be shown that $\ymin(t)\geq \triangle$ for all $t\geq T_1$. 

During $[0,T_1]$, the service rate of all sub-queues remain zero; however, when $n_0<+\infty$, $T_1$ is finite and for the computation of load factor this time interval can be neglected. Therefore,  the load factor for each sub-queue will be the same as the sub-linear case \eqref{eq:load-factor-upper-bound}. In this case, however, in order to have $\rho<1$, we get the counterpart of \eqref{eq:triangle-star} as:
\begin{equation}\label{eq:triangle-star-superlinear}
\triangle > \triangle^*(\lambda).
\end{equation}
It should be noted that since the batch release control policy iteratively releases from odd and even sub-queues, we need at least two sub-queues to be able to implement this policy. As a result, $\triangle$ cannot be arbitrary large and $\triangle<\ell/2$. This constraint gives the following  bound on the admissible throughput under this policy
\begin{equation}\label{eq:lambda-star-superlinear}
\lambda<\lambda^*:=(\ell/2)^{m-1}/2RF
\end{equation}
 The following result shows that for the above range of throughput, the queue length in HTQ1, $N_1(t)$, remains bounded at all times. 
  \begin{proposition}
\label{prop:pi1-policy-super-linear}
For any $\lambda<\lambda^*$, $\triangle\in\big(\triangle^*(\lambda),\ell/2\big]$, $\varphi \in \Phi_F$, $F>0$, $\psi \in \Psi$ with $\supp(\psi)=[0,R]$, $R>0$, $m>1$, $x_0\in[0,L]^{n_0}$, $L>0$, $n_0\in \mathbb{N}$, $N_1(t)$ is bounded for all $t \geq 0$ under $\pi_{\triangle}^b$, almost surely.
\end{proposition}

\begin{proof}
The proof is similar to proof of Proposition \ref{prop:pi1-policy}. In particular, by \eqref{eq:triangle-star-superlinear} and \eqref{eq:lambda-star-superlinear}, one can show that load factor \eqref{eq:load-factor-upper-bound} remains strictly smaller than one. This implies that no sub-queue in HTQ1 can grow unbounded, and $N_1(t)$ remains bounded for all times, with probability one. 
\qed
\end{proof}


\begin{proposition}
\label{eq:waiting-time-super-linear}
For any $\lambda<\lambda^*$, $\varphi \in \Phi_F$, $F>0$, $\psi \in \Psi$, $m> 1$, the average waiting time in HTQ1 under $\pi_{\triangle}^b$ for $\triangle = \ell/2$ is upper bounded as: 
\begin{equation}
\label{eq:waiting-time-super-linear-eq}
W \leq \frac{2R}{(\ell/2)^m}+\frac{R}{(\ell/2)^m}\frac{2R\lambda F(\ell/2)^{1-m}}{1-2R\lambda F(\ell/2)^{1-m}}
\end{equation}
\end{proposition}
\begin{proof}
The proof is very similar to the proof of Proposition \ref{eq:waiting-time}. Thus, we get the following bounds:
\begin{align*}
W \leq \frac{2R}{\triangle^m} + \frac{R}{\triangle^m} \frac{\rho}{1-\rho}  \leq \frac{2R}{\triangle^m}+\frac{R}{\triangle^m}\frac{2R\lambda F\triangle^{1-m}}{1-2R\lambda F\triangle^{1-m}}
\end{align*} 
The right hand side of the above inequality is a decreasing function of $\triangle$; therefore, $\triangle = \ell/2$ minimizes it, and gives \eqref{eq:waiting-time-super-linear-eq}. 
\qed
\end{proof}

Let $\hat{W}(m,F,R,\eta)$ 
be the value of $\lambda$ for which the right hand side of \eqref{eq:waiting-time-super-linear-eq} is equal to $\eta$, if such a $\lambda\leq \lambda^*$ exists, and let it be equal to $\lambda^*$ otherwise. Note that since the right hand side of \eqref{eq:waiting-time-super-linear-eq} is monotonically increasing in $\lambda$, for all $\lambda\leq \hat{W}(m,F,R,\eta)$ the introduced perturbation remains upper bounded by $\eta$. 
\begin{theorem}\label{thm:main-batch-super-linear}
For any $\varphi \in \Phi_F$, $F>0$, $\psi \in \Psi$ with $\text{supp}(\psi)=[0,R]$, $R>0$, $m >1$, $x_0 \in [0,L]^{n_0}$, $n_0 \in \NN$, $L>0$, and maximum permissible perturbation $\eta \geq 0$, 
$$
\lambdamax(L,m,\varphi,\psi,x_0,\delta=0,\eta) \geq \hat{W}(m,F,R,\eta)
.$$
\end{theorem}
\begin{proof}
For any $\lambda <\hat{W}(m,F,R,\eta)$, under $\pi_\triangle^b$, Lemma \ref{eq:successive-release-time-difference} and Proposition \ref{prop:pi1-policy-super-linear} imply that, for finite $n_0$, $N_2(t)$ and $N_1(t)$ remain bounded for all times, with probability one. Also, by Proposition \ref{eq:waiting-time-super-linear} and the definition of $\hat{W}(m,F,R,\eta)$, the introduced perturbation remains upper bounded by $\eta$. 
\qed
\end{proof}

\section{Simulations}
\label{sec:simulations}
In this section, we present simulation results on throughput analysis, and compare with our theoretical results from previous sections.

\begin{figure}
\centering
\subfigure[]{\includegraphics[width=5cm]{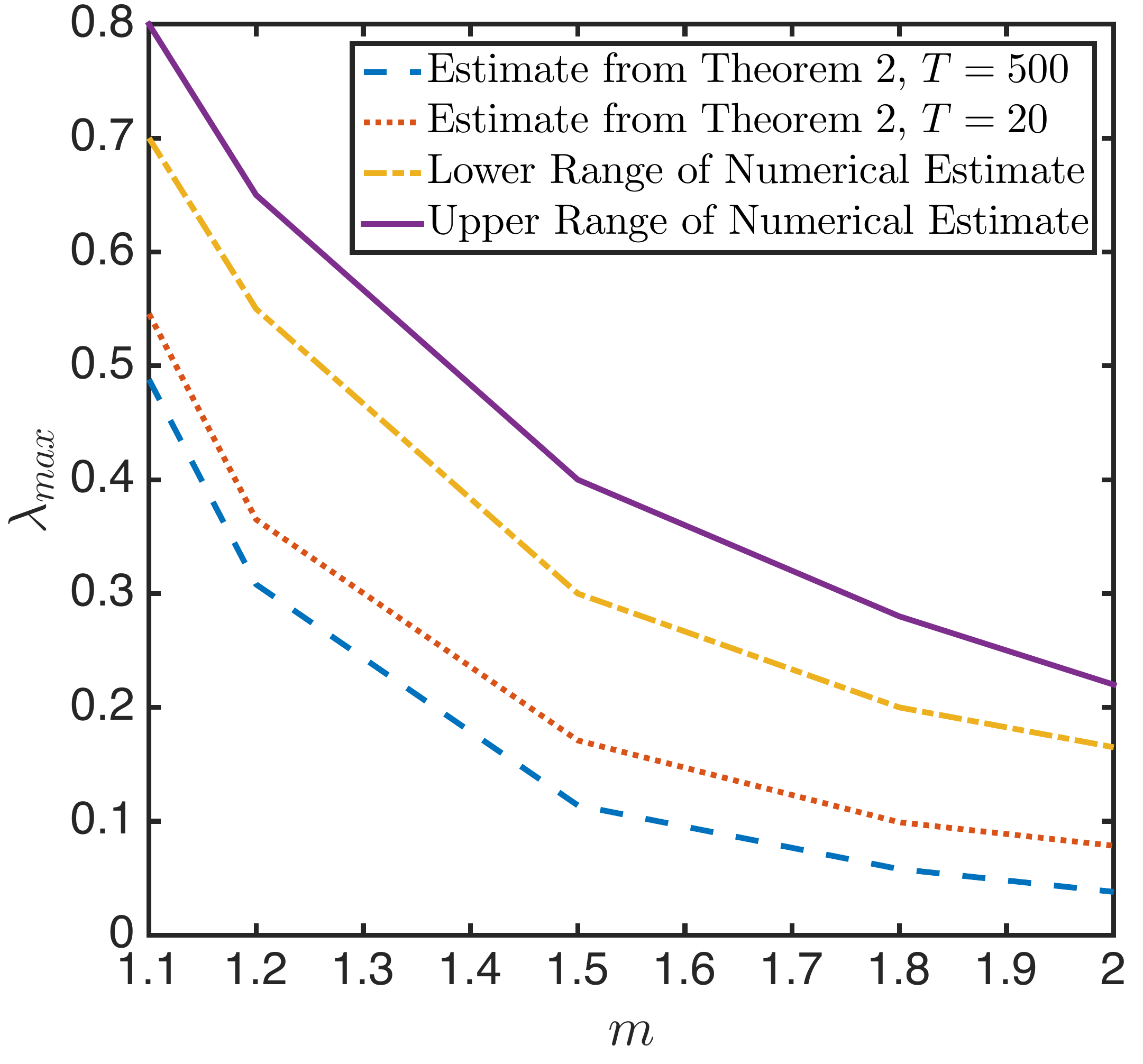}}
\centering
\hspace{0.3in}
\subfigure[]{\includegraphics[width=5cm]{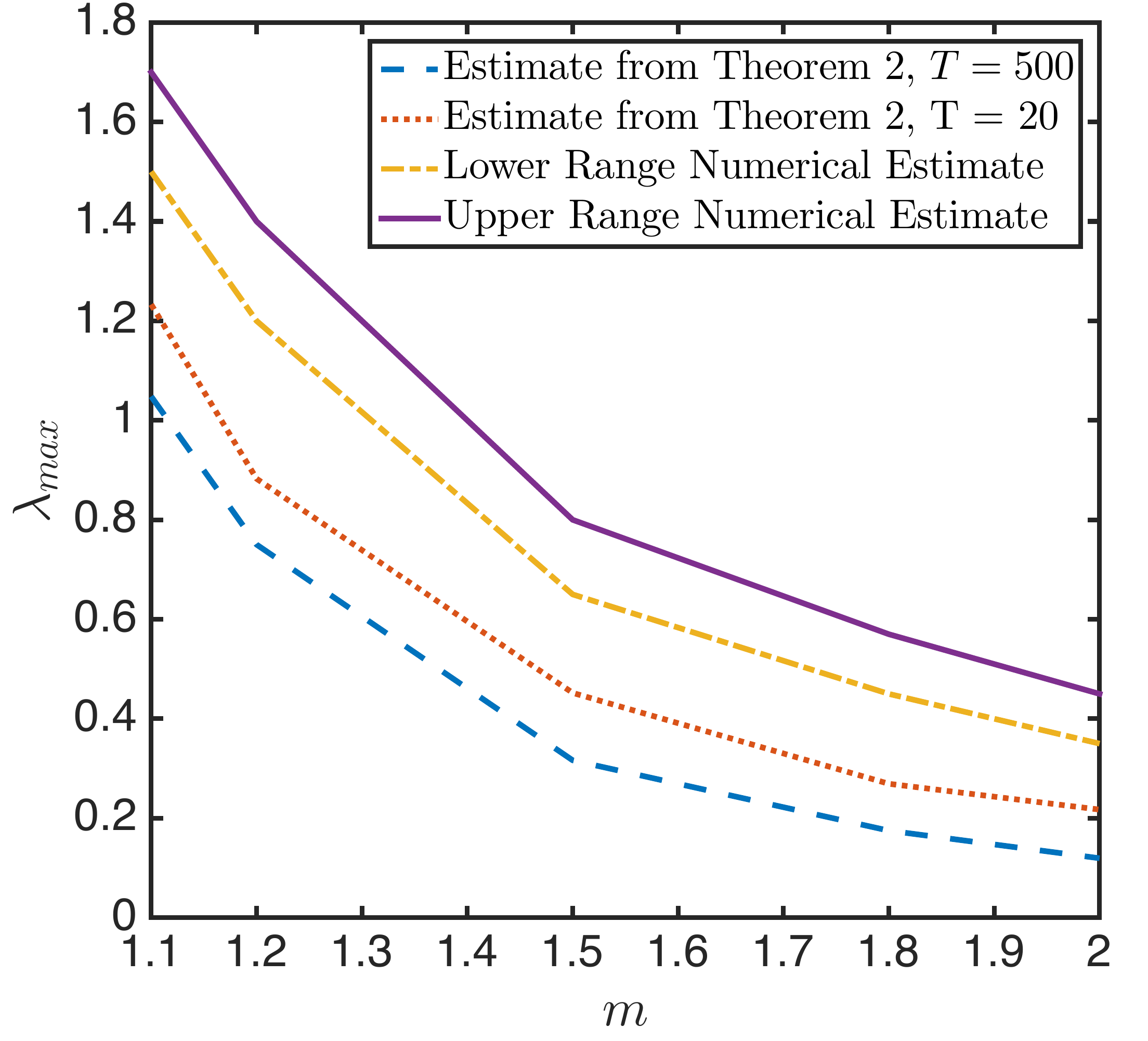}}
\caption{Comparison between theoretical estimates of throughput from Theorem~\ref{thm:superlinear-bound-empty}, and range of numerical estimates from simulations, for zero initial condition. The parameters used for this case are: $L = 1$, $\delta = 0.1$, and (a) $\varphi = \delta_0$, $\psi = \delta_L$, (b)  $\varphi = U_{[0,L]}$, $\psi=U_{[0,L]}$.
 \label{fig:lambda-m-empty}}
\end{figure}

 \begin{figure}[htb!]
\begin{center}
\includegraphics[width=5cm]{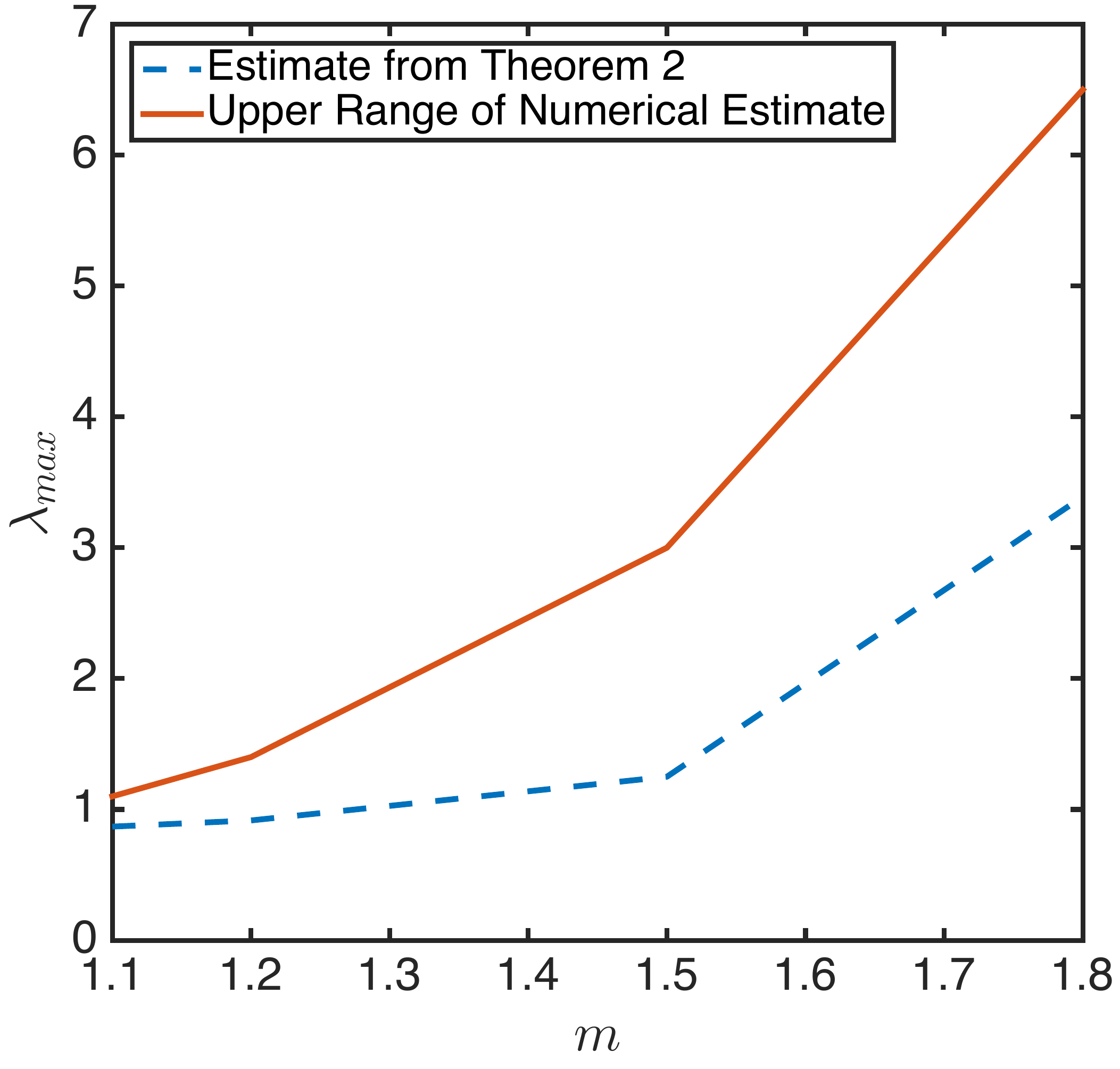} 
\caption{Comparison between theoretical estimates of throughput from Theorem  \ref{thm:superlinear-bound-empty}, and range of numerical estimates from simulations, for zero initial condition. The parameters used for this case are: $L = 100$, $\delta = 0.1$, $T=10$, and $\varphi = \delta_0$, $\psi = \delta_L$.}
\label{fig:superlinear-large-L}
\end{center}
\end{figure} 

 \begin{figure}[htb!]
\begin{center}
\includegraphics[width=5cm]{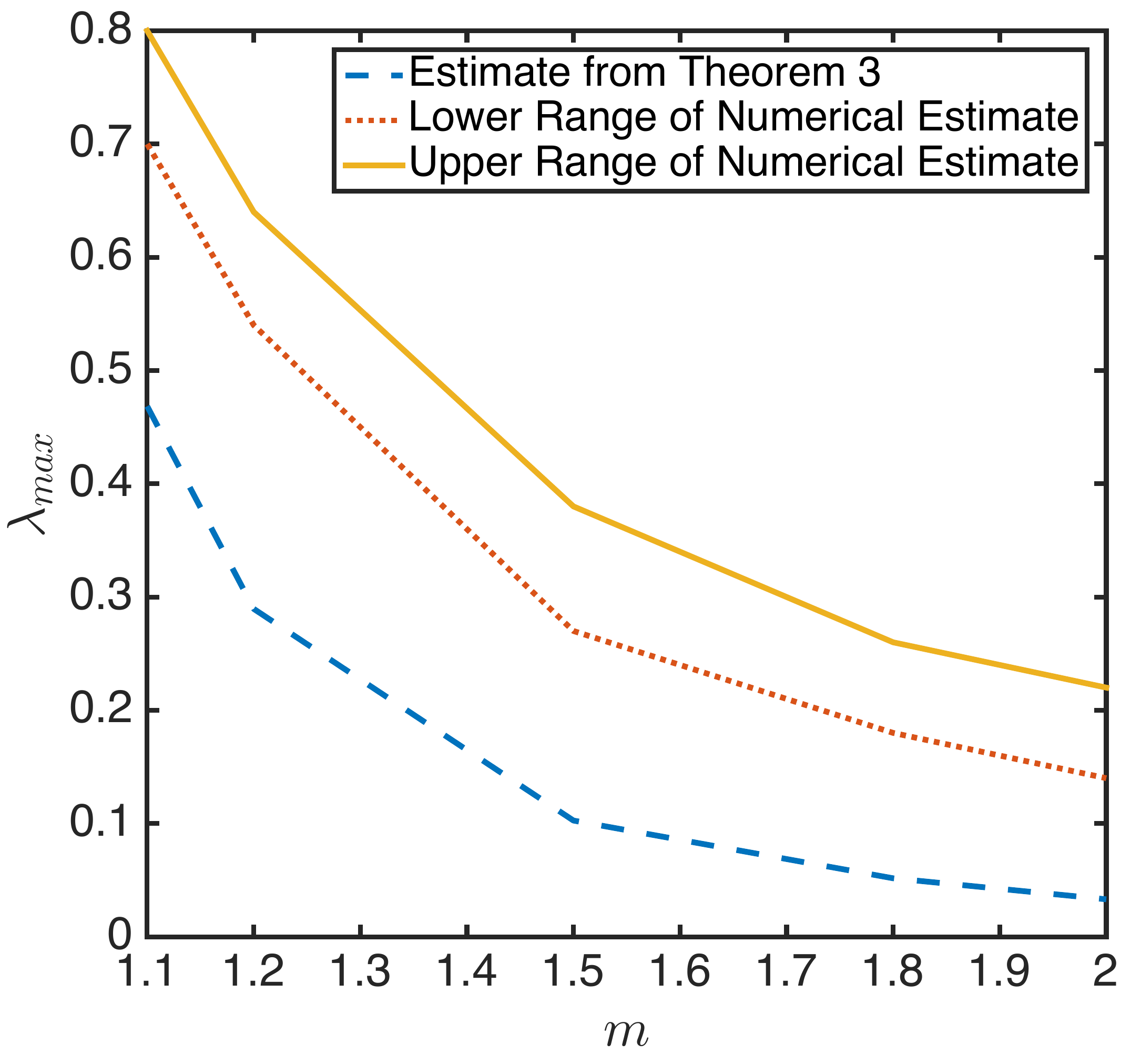} 
\caption{Comparison between theoretical estimates of throughput from Theorem~\ref{thm:superlinear-bound-non-empty}, and range of numerical estimates from simulations.  The parameters used for this case are: $L = 1$, $\delta = 0.1$, $\varphi = \delta_0$,  $\psi = \delta_L$, $w_0=1$ and $n_0=4$, $x_1(0) =0.6, x_2(0) =0.7, x_3(0) =0.8, x_4(0) =0.9$.}
\label{fig:phase-transition-non-empty}
\end{center}
\end{figure}


Figures~\ref{fig:lambda-m-empty}, \ref{fig:superlinear-large-L} and \ref{fig:phase-transition-non-empty} show comparison between the lower bound on throughput over finite time horizons, as given by Theorems~\ref{thm:superlinear-bound-empty} and \ref{thm:superlinear-bound-non-empty}, and the corresponding numerical estimates from simulations. 
 Figures~\ref{fig:lambda-m-empty} and \ref{fig:superlinear-large-L} are for zero initial condition, and Figure~\ref{fig:phase-transition-non-empty} is for non-zero initial condition.

Figures~\ref{fig:sublinear-lowerbound} and \ref{fig:superlinear-batch} show comparison between the lower bound on throughput as given by the bacth release control policy, as per Theorems~\ref{thm:main-sub-linear} and \ref{thm:main-batch-super-linear}, respectively, under a couple of representative values of maximum permissible perturbation $\eta$. In particular, Figure~\ref{fig:sublinear-lowerbound} demonstrates that the lower bound achieved from Theorem \ref{thm:main-sub-linear} increases drastically as $m \to 0^+$. Both the figures also confirm that the throughput indeed increases with increasing maximum permissible perturbation $\eta$.

It is instructive to compare Figures~\ref{fig:lambda-m-empty}(b) and \ref{fig:superlinear-batch}(a), both of which depict throughput estimates for the sub-linear case, however obtained from different methods, namely busy period distribution and batch release control policy. Accordingly, one should bear in mind that the two bounds have different qualifiers attached to them: the bound in Figure~\ref{fig:lambda-m-empty}(b) is valid probabilistically only over a finite time horizon, whereas the bound in Figure~\ref{fig:superlinear-batch}(a) is valid with probability one, although under a perturbation to the arrival process. 

 \begin{figure}[htb!]
\begin{center}
\includegraphics[width=5cm]{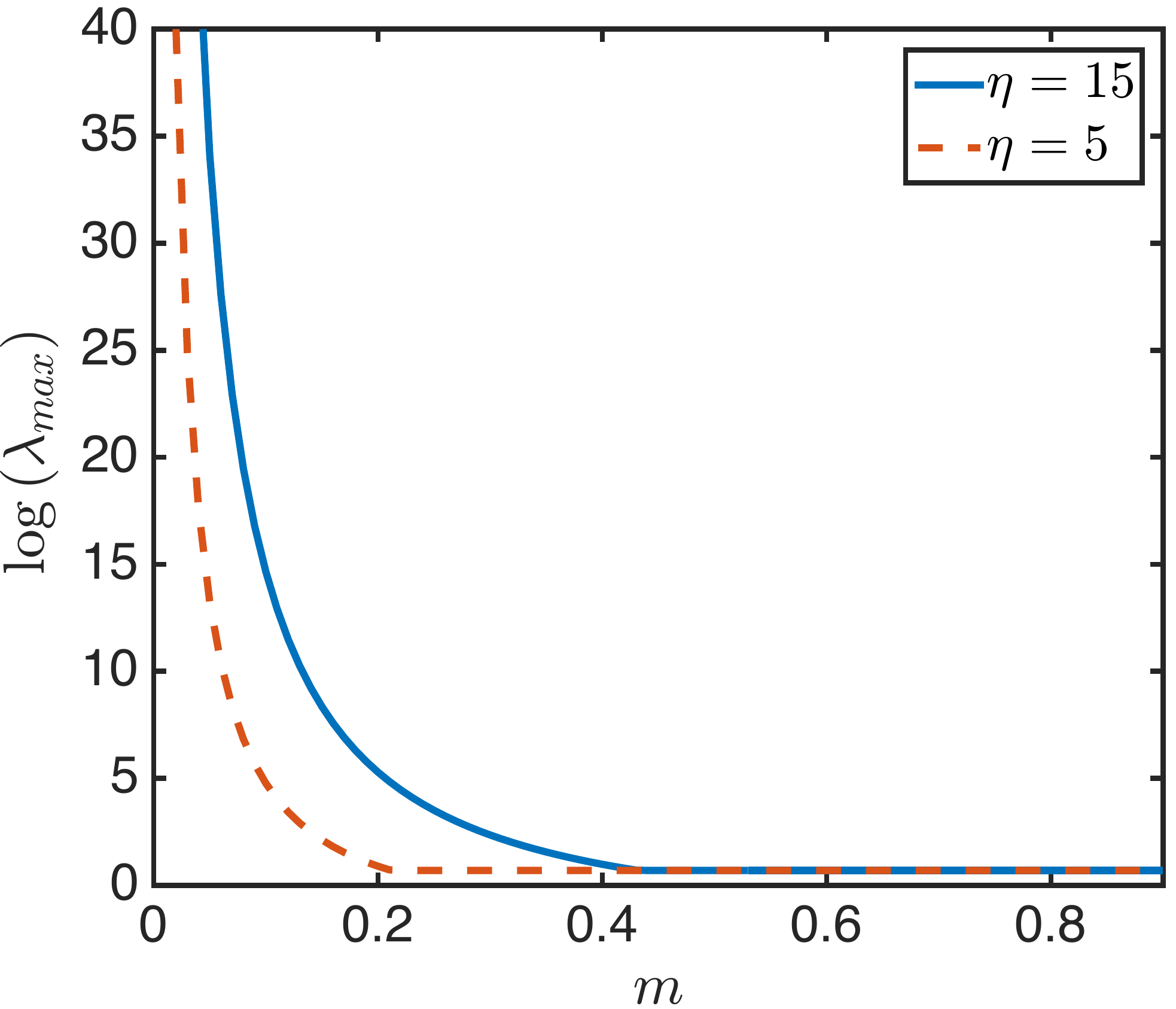} 
\caption{Theoretical estimates of throughput from Theorems~\ref{thm:main-sub-linear} for different values of $\eta$.  The parameters used for this case are: $L = 1$, $\varphi = U_{[0,L]}$,  $\psi = U_{[0,L]}$, and $w_0=0$ . Note that the vertical axis is in logarithmic scale.}
\label{fig:sublinear-lowerbound}
\end{center}
\end{figure}



\begin{figure}
\centering
\subfigure[]{\includegraphics[width=5cm]{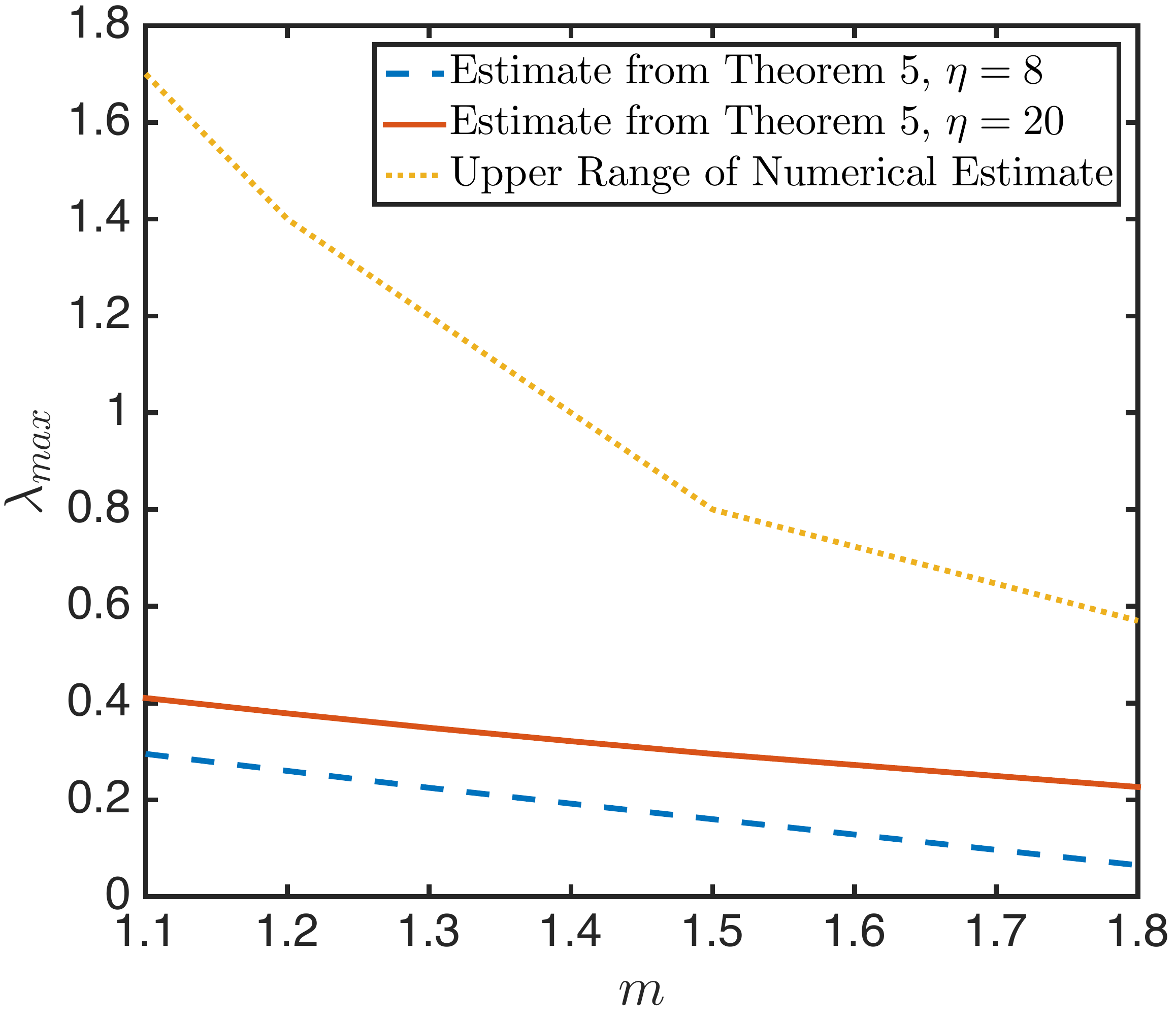}}
\centering
\hspace{0.3in}
\subfigure[]{\includegraphics[width=5cm]{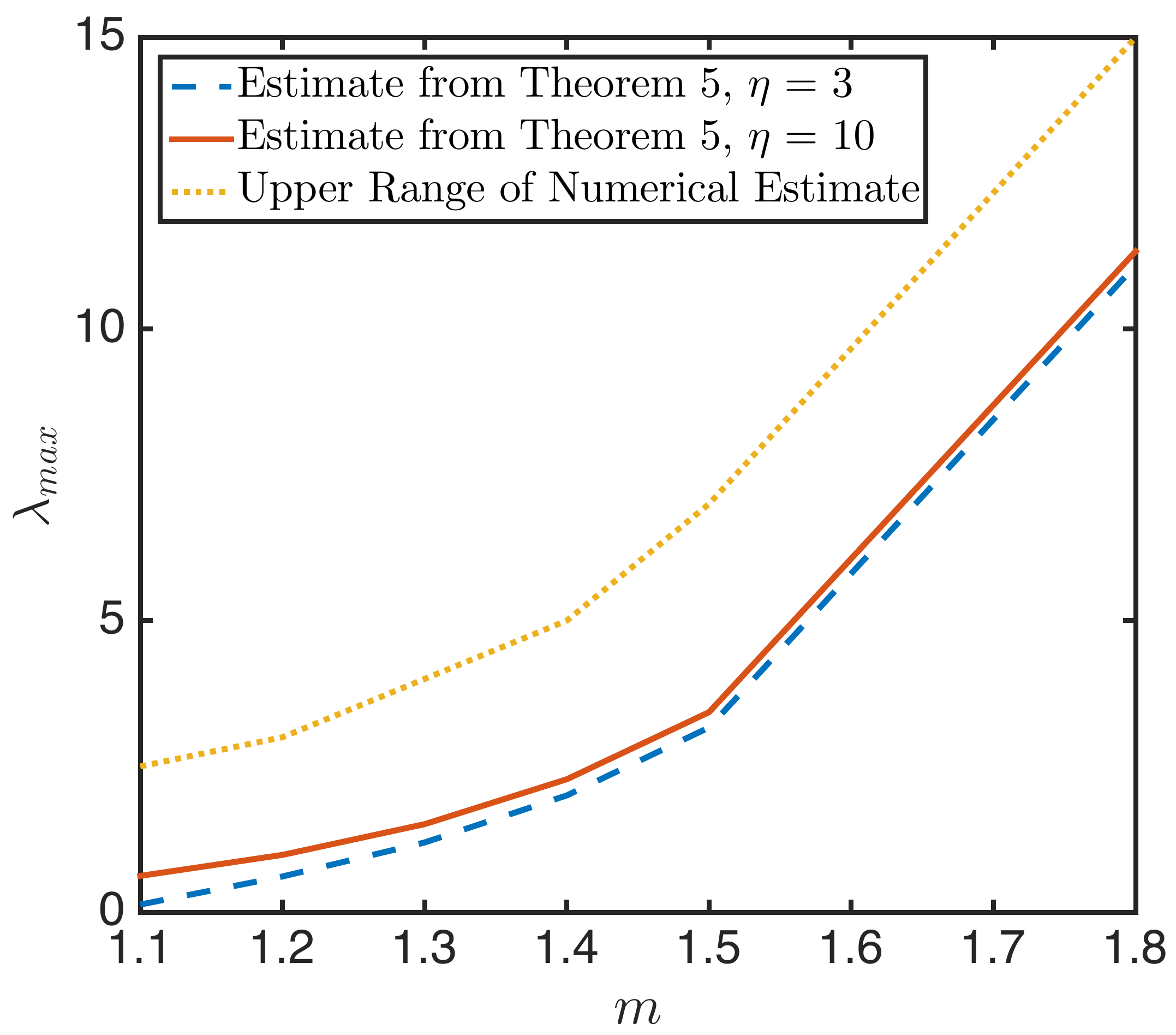}}
\caption{Theoretical estimates of throughput from Theorem \ref{thm:main-batch-super-linear}, and numerical estimates from simulations for different  The parameters used for this case are: $\varphi=U_{[0,L]}$, $\psi=U_{[0,L]}$, and (a) $L = 1$, (b) L = 100. \label{fig:superlinear-batch}}
\end{figure}

 \begin{figure}[htb!]
\begin{center}
\includegraphics[width=5cm]{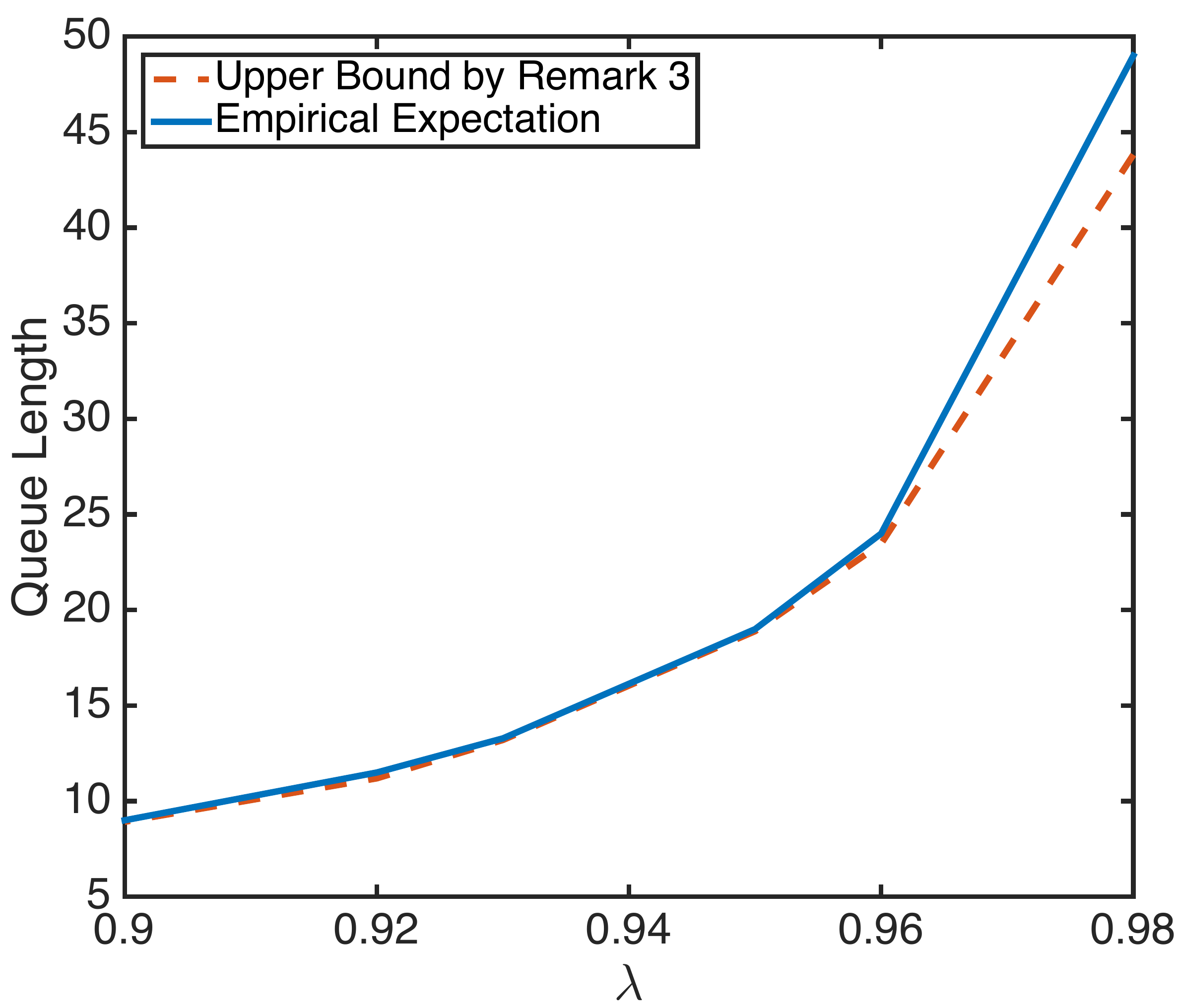} 
\caption{Comparison between the empirical expectation of the queue length and the upper bound suggested by Remark~\ref{remark:waiting-time}. We let the simulations run up to time $t=80,000$. The parameters used for this case are: $L=1$, $m=1$, $\varphi=\delta_0$, $\psi=\delta_L$. For these values, we have $\lambdamax=1$.}
\label{fig:queue-length-linear}
\end{center}
\end{figure}  

Finally, Figure~\ref{fig:queue-length-linear} shows a good agreement between queue length bound suggested by Remark~\ref{remark:waiting-time}, and the corresponding numerical estimates in the linear case.

\section{Conclusions}
\label{sec:conclusions}

In this paper, we formulated and analyzed a novel horizontal traffic queue. A key characteristic of this queue is the state dependence of its service rate. We establish useful properties of the service rate dynamics. We also extend calculations for M/G/1 busy period distributions to our setting, even for non-empty initial condition. These results allow us to provide tight results for throughput in the linear case, and probabilistic bounds on queue length over finite time horizon in the super-linear case. We also study throughput under a batch release control policy, where the additional waiting induced by the control policy is interpreted as a perturbation to the arrival process. We provide lower bound on the throughput for a maximum permissible perturbation. In particular, if the allowable perturbation is sufficiently large, then this lower bound grows unbounded as $m \to 0^+$. Simulation results suggest a sharp phase transition in the throughput as the car-following behavior transitions from super-linear to sub-linear regime. 

In future, we plan to sharpen our analysis to theoretically demonstrate the phase transition behavior. 
This could include, for example, considering other release control policies with better perturbation properties. The increasing nature of the throughput in the super-linear regime for large values of $L$, as illustrated in Figure~\ref{fig:throughput-simulations}, is possibly because the car-following model considered in this paper does not impose any explicit upper bounds on the speed of the vehicles. We plan to extend our analysis to such practical constraints, as well as to higher order, e.g., second order, car-following models, and models emerging from consideration of inter-vehicle distance beyond the vehicle immediately in front. The connections with processor sharing queue, as highlighted in this paper, suggest the possibility of utilizing the construct of measure-valued state descriptors~\cite{Grishechkin:94,Gromoll.Puha.ea:02} to derive fluid and diffusion limits of the proposed horizontal traffic queues. In particular, one could interpret the measure-valued state descriptor to play the role of traffic density in the context of traffic flow theory. Along this direction, we plan to investigate connections between the fluid limit of horizontal traffic queues, and PDE models for traffic flow. 

\bibliographystyle{plain}
\bibliography{./ksmain,./savla}

\section{Appendix}
\label{sec:appendix}

In this section, we gather a few technical results that are used in the main results of the paper. 

%

\begin{definition}[Type $K$ function]~\cite{Smith:08}
Let $g:S\mapsto\mathbb{R}^n$ be a function on $S\subset\mathbb{R}^n$. $g$ is said to be of type $K$ in $S$ if, for each $i \in \until{n}$, $g_i(z_1)\leq g_i(z_2)$ holds true for any two points $z_1$ and $z_2$ in $S$ satisfying $z_1 \leq z_2$ and $z_{1,i}=z_{2,i}$.
\end{definition}

\begin{lemma}
\label{lem:monotonicity-same-size}
Let $\map{g}{S}{\RR^N}$ and $\map{h}{S}{\RR^N}$ be both of type $K$ over $S \subset \RR^N$. Let $z_1(t)$ and $z_2(t)$
be the solutions to $\dot{z}=g(z)$ and $\dot{z}=h(z)$, respectively, starting from initial conditions $z_1(0)$ and $z_2(0)$ respectively. 
Let $S$ be positively invariant under $\dot{z}=g(z)$ and $\dot{z}=h(z)$.
If $g(z) \leq h(z)$ for all $z \in S$, and $z_1(0) \leq z_2(0)$, then $z_1(t) \leq z_2(t)$ for all $t \geq 0$. 
\end{lemma}
\begin{proof}
By contradiction, let $\tilde{t} \geq 0$ be the smallest time at which, there exists, say $k \in \until{N}$, such that $z_1(\tilde{t}) \leq z_2(\tilde{t})$, $z_{1,k}(\tilde{t}) = z_{2,k}(\tilde{t})$, and  
\begin{equation}
\label{eq:class-K-contradiction}
g_k(z_1(\tilde{t})) > h_k(z_2(\tilde{t})).
\end{equation}
Since $g(z)$ is of class $K$, $z_1(\tilde{t}) \leq z_2(\tilde{t})$ and $z_{1,k}(\tilde{t}) = z_{2,k}(\tilde{t})$ imply that $g(z_1(\tilde{t})) \leq g(z_2(\tilde{t}))$. This, combined with the assumption that $g(z) \leq h(z)$ for all $z \in S$ implies that $g(z_1(\tilde{t})) \leq h(z_2(\tilde{t}))$, which contradicts \eqref{eq:class-K-contradiction}.
\qed
\end{proof}

Lemma~\ref{lem:monotonicity-same-size} is relevant because the basic dynamical system in our case is of type $K$. 

\begin{lemma}
\label{lem:car-following-type-K}
For any $L > 0$, $m>0$, and $N \in \NN$, the right hand side of \eqref{eq:x-dynamics-in-Rn} is of type $K$ in $\RR_+^N$.
\end{lemma}
\begin{proof}
Consider $\tilde{x}, \hat{x} \in \RR_+^N$ such that $\tilde{x} \leq \hat{x}$. 
If $\tilde{x}_i = \hat{x}_i$ for some $i \in \until{N}$, then, according to \eqref{eq:inter-vehicle-distance-Rn}, $y_i(\tilde{x})-y_i(\hat{x})= \left(\tilde{x}_{i+1} - \hat{x}_{i+1} \right) -\left(\tilde{x}_i-\hat{x}_i \right) = \tilde{x}_{i+1} - \hat{x}_{i+1}$ if $i \in \until{N-1}$, and is equal to $\left(\tilde{x}_{1} - \hat{x}_{1} \right) - \left(\tilde{x}_N-\hat{x}_N \right)= \tilde{x}_{1} - \hat{x}_{1}$ if $i=N$. In either case, $y_i(\tilde{x}) \leq y_i(\hat{x})$, which also implies $y_i^m(\tilde{x}) \leq y_i^m(\hat{x})$ for all $m > 0$.
\qed
\end{proof}


%
%

In order to state the next lemma, we need a couple of additional definitions. 

\begin{definition}[Monotone Aligned and Monotone Opposite Functions]
Two strictly monotone functions $\map{h}{\RR}{\RR}$ and $\map{g}{\RR}{\RR}$ are said to be \emph{monotone-aligned} if they are both either strictly increasing, or strictly decreasing. Similarly, the two functions are called \emph{monotone opposite} if one of them is strictly increasing, and the other is strictly decreasing. 
\end{definition}


\begin{lemma}
\label{lem:appendix-general-summation}
Let $\map{h}{\RR_+}{\RR}$ and $\map{g}{\RR_+}{\RR}$ be strictly monotone functions. Then, for every $y \in \mc S_N^L$, $n \in \NN$, $L>0$,
\begin{equation}
\label{eq:summation-general}
\sum_{i=1}^N h(y_i) \left(g(y_{i+1}) - g(y_i) \right)
\end{equation}
is non-negative if $h$ and $g$ are monotone-opposite, and is non-positive if $h$ and $g$ are monotone-aligned. Moreover, \eqref{eq:summation-general} is equal to zero if and only if $y=\frac{L}{N} \onebf$.
\end{lemma}
\begin{proof}
For $i \in \until{N}$, let $I_i$ be the interval with end points $g(y_i)$ and $g(y_{i+1})$. 
For $i \in \until{N}$, let $f_i(z):=\sgn{g(y_{i+1}) - g(y_i)} h(y_i) \onebf_{I_i}(z)$. 
Let $g_{\text{min}}:=\min_{i \in \until{N}} g(y_i)$, and $g_{\text{max}}:=\max_{i \in \until{N}} g(y_i)$. With $f(z):=\sum_{i=1}^N f_i(z)$,
 \eqref{eq:summation-general} can then be written as:
\begin{equation}
\label{eq:line-integral}
\sum_{i=1}^N h(y_i) \left(g(y_{i+1}) - g(y_i) \right) = \int_{g_{\text{min}}}^{g_{\text{max}}} f(z) \, dz.
\end{equation}
We now show that, for every $z \in [g_{\text{min}}, g_{\text{max}}] \setminus \{g(y_i): i \in \until{N}\}$,  $f(z)$ is non-negative if $h$ and $g$ are monotone-opposite, and is non-positive if $h$ and $g$ are monotone-aligned. 
This, together with \eqref{eq:line-integral}, will then prove the lemma.

It is easy to see that every $z \in [g_{\text{min}}, g_{\text{max}}] \setminus \{g(y_i): i \in \until{N}\}$ belongs to an even number of intervals in $\{I_i: \, i \in \until{N}\}$, say $I_{\ell_1}, I_{\ell_2}, \ldots$, with $\ell_1 < \ell_2 < \ldots$ (see Figure \ref{fig:intervals} for an illustration).
We now show that $f_{\ell_1}(z) + f_{\ell_2}(z)$ is non-negative if $h$ and $g$ are monotone-opposite, and is non-positive if $h$ and $g$ are monotone-aligned. The same argument holds true for $f_{\ell_3}(z)+f_{\ell_4}(z), \ldots$. 
Assume that $g(y_{\ell_1}) \leq g(y_{\ell_2})$; the other case leads to the same conclusion. By definition of $f_i$'s, $f_{\ell_1}(z)=h(y_{\ell_1})$ and $f_{\ell_2}(z)=-h(y_{\ell_2})$. $g(y_{\ell_1}) \leq g(y_{\ell_2})$ implies that 
$f_{\ell_1}(z)+f_{\ell_2}(z)=h(y_{\ell_1})-h(y_{\ell_2})$ is non-negative if $h$ and $g$ are monotone-opposite, and is non-positive if $h$ and $g$ are monotone-aligned, with the equality holding true if and only if $y_{\ell_1}=y_{\ell_2}$. 
\begin{figure}[htb!]
 \centering
 \includegraphics[width=15.5cm]{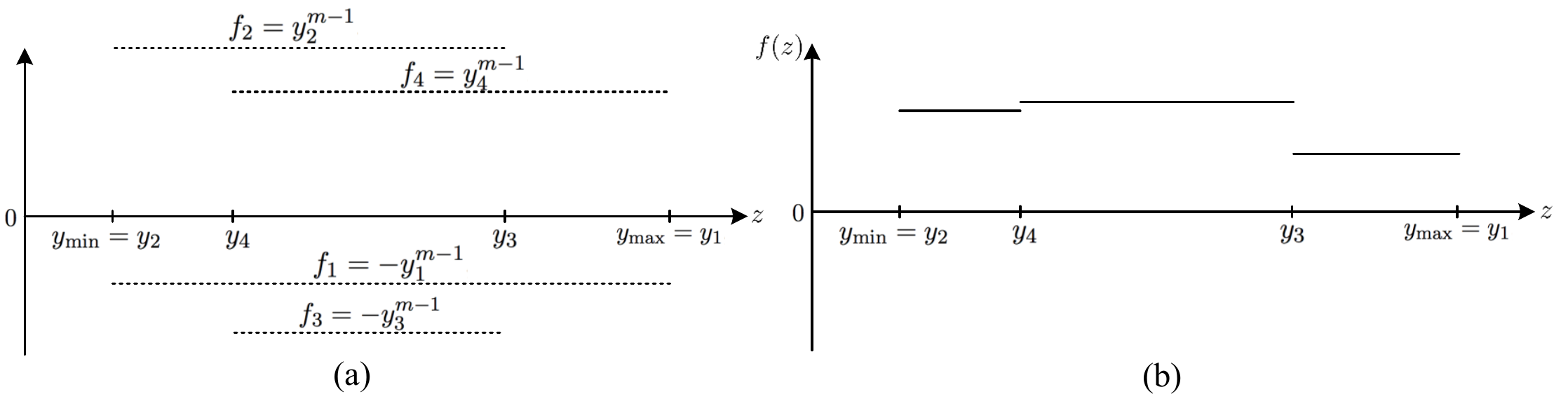} 
  \caption{A schematic view of (a) $f_i(z), i=\{1,2,3,4\}$ and (b) $f(z)=\sum_{i=1}^4f_i(z)$ for a $y\in\mc S_4^L$ ($L=1$) with $\ymin = y_2 < y_4 < y_3 < y_1 = \ymax$ for a $m<1$. }
    \label{fig:intervals}
\end{figure}
\qed
\end{proof}

%


\begin{lemma}\label{lemma:t-n}
For $n \in \NN \setminus \{1\} $, let $\psi_n$ be the n-fold convolution of $\psi \in \Psi$. Then, 
$$\int_{0}^t z \, \psi(z)\psi_{n-1}(t-z)\de z = \frac{t}{n}\psi_n(t) \qquad \forall \, t \geq 0$$
\end{lemma}
\begin{proof}
Let $J_1,\cdots,J_n$ be $n$ random variables, all with distribution $\psi$. Therefore, the probability distribution function of the random variable $V:=\sum_{i=1}^n J_i$ is $\psi_n$. 
Using linearity of the expectation, we get that 
\begin{equation*}
t=E\left[\sum_{i=1}^n J_i|V=t\right]=\sum_{i=1}^n E\left[J_i|V=t\right]=n \,  E\left[J_1|V=t \right]
\end{equation*}
%
i.e., 
\begin{equation}
\label{eq:D1-cond-sum}
E\left[J_1|V=t \right] = \frac{t}{n}
\end{equation}
Let $f_{J_1|V}(j_1|t)$ denote the probability distribution function of $J_1|V$. By definition:
\be\label{eq:f-d1-z}
f_{J_1|V}(j_1|t) = \frac{f_{J_1,V}(j_1,t)}{\psi_n(t)}=\frac{\psi(j_1)\psi_{n-1}(t-j_1)}{\psi_n(t)}
\ee
Therefore, using \eqref{eq:D1-cond-sum} and \eqref{eq:f-d1-z}, we get that
\begin{equation*}
E[J_1|V=t] = \int_0^t zf_{J_1|V}(z|t)\, \de z = \int_0^t z\frac{\psi(z)\psi_{n-1}(t-z)}{\psi_n(t)} \, \de z = \frac{t}{n}
\end{equation*}
Simple rearrangement gives the lemma. 
\qed
\end{proof}

The following is an adaptation of \cite[Lemma 2.3.4]{Ross:96}. 
\begin{lemma}\label{lemma:busy-period-type1-prob}
Let $a_1, \cdots, a_{n-1}$ denote the ordered values from a set of $n-1$ independent uniform $(0,t)$ random variables. Let $\tilde d_0=z \geq 0$ be a constant and $\tilde d_1, \tilde d_2, \cdots \tilde d_{n-1}$ be i.i.d. non-negative random variables that are also independent of $\{a_1, \cdots, a_{n-1}\}$, then 
\begin{align*}
\Pr(\tilde d_{k}+ \cdots +\tilde d_{n-1}\leq a_{n-k},k=1,\cdots,n-1 |\tilde d_0+ \cdots +\tilde d_{n-1}=t, \tilde{d}_0=z)
 = \begin{cases} z/t & z<t \\
                          0      & \text{otherwise}
    \end{cases}
\end{align*}
\end{lemma}


\end{document}